\documentclass[hidelinks,onefignum]{siamart220329}

\usepackage{amssymb}
\usepackage{amsfonts}
\usepackage{amsmath}
\usepackage[toc,page]{appendix}

\newtheorem{example}{Example}

\usepackage{lipsum}
\usepackage{amsfonts}
\usepackage{graphicx}
\usepackage{epstopdf}
\usepackage{algorithmic}
\ifpdf
  \DeclareGraphicsExtensions{.eps,.pdf,.png,.jpg}
\else
  \DeclareGraphicsExtensions{.eps}
\fi

\newsiamremark{remark}{Remark}
\newsiamremark{hypothesis}{Hypothesis}
\crefname{hypothesis}{Hypothesis}{Hypotheses}
\newsiamthm{claim}{Claim}

\headers{Towards the Buchanan-Lillo Conjecture}{E. Braverman and J.\,I.~Stavroulakis}

\title{Towards a resolution of the Buchanan-Lillo Conjecture \thanks{Submitted to the editors on August 10, 2023. \funding{The first author was supported by the NSERC Grant RGPIN-2020-03934. 
The second author acknowledges the support of Ariel University. }}}

\author{Elena Braverman  \thanks{Dept. of Math. and Stats., University of
Calgary, Calgary, AB  T2N 1N4, Canada 
  (\email{maelena@ucalgary.ca},\email{maelena@math.ucalgary.ca}),
\url{https://science.ucalgary.ca/mathematics-statistics/contacts/elena-braverman}).}
\and  John Ioannis Stavroulakis  \thanks{Department of Mathematics, Ariel University, 
Ariel 4076414, Israel (\email{john.ioannis.stavroulakis@gmail.com},\email{Ioannis.Stavroul@msmail.ariel.ac.il})
}}

\usepackage{amsopn}

\begin{document}

\maketitle

\begin{abstract}
Buchanan and Lillo both conjectured that oscillatory solutions of the first-order delay differential equation with positive feedback
$x^{\prime }(t)=p(t)x(\tau (t))$, $t\geq 0$, 
where $0\leq p(t)\leq 1$, $0\leq t-\tau (t)\leq 2.75+\ln2,t\in \mathbb{R},$ are asymptotic to a shifted multiple of a unique periodic solution. This special solution was known to be uniform for all nonautonomous equations, and intriguingly, can also be described from the more general perspective of the mixed feedback case (sign-changing $p$). The analog of this conjecture for negative feedback, $p(t)\leq0$, was resolved by Lillo, and the mixed feedback analog was recently set as an open problem. 
In this paper, we investigate the convergence properties of the special periodic solutions in the mixed feedback case, characterizing the threshold between bounded and unbounded
oscillatory solutions, with standing assumptions that $p$ and $\tau$ are measurable, $\tau (t)\leq t$ and $\lim_{t\rightarrow \infty }\tau (t)=\infty$. 
We prove that nontrivial oscillatory solutions on this threshold are asymptotic (differing by $o(1)$) to the special periodic solutions for mixed feedback, which include the periodic solution of the positive feedback case. The conclusions drawn from these results elucidate and refine the conjecture of Buchanan and Lillo that
oscillatory solutions in the positive feedback case $p(t)\geq0$, would differ from a multiple, translation, of the special periodic solution, by $o(1)$.
\end{abstract}





\section{Introduction}

Consider the nonautonomous first-order delay equation%
\begin{equation}
x^{\prime }(t)=p(t)x(\tau (t)),~~t\geq t_{0},  \label{06.18}
\end{equation}%
where $p:%
\mathbb{R}
\rightarrow 
\mathbb{R}
,\tau :%
\mathbb{R}
\rightarrow 
\mathbb{R}
$ are measurable, $p$ is locally integrable, and $\tau (t)\leq t,\forall t\in 
\mathbb{R}$
, $\lim_{t\rightarrow \infty }\tau (t)=\infty .$ An absolutely continuous
function $x:[\inf \{\tau (t),t\geq t_{0}\},+\infty )\rightarrow 
\mathbb{R}
,$ satisfying \eqref{06.18} on $[t_{0},+\infty )$ is a \textit{solution} of %
\eqref{06.18}. The sign of $p$ is referred to as the \textit{feedback} of
the equation. A function $x:A\rightarrow 
\mathbb{R}
,A\subset 
\mathbb{R}
$, is called \textit{oscillatory}, if and only if it possesses arbitrarily
large zeros. In such a case, a nonoscillation interval $(a,b),$ where $a<b,$
and $x(t)\neq 0,t\in (a,b)$ is a \textit{semicycle}.

The ordinary differential equation 
\begin{equation*}
x^{\prime }(t)=p(t)x(t),~t\geq 0
\end{equation*}%
possesses a one dimensional solution family, multiples of a single positive
solution, whose asymptotics are determined by the sign and behavior of $%
\int_{0}^{t}p(s)ds$ as $t\rightarrow \infty $.  Introducing a functional
argument as in \eqref{06.18}, the phase space becomes infinite dimensional \cite[%
Chapters 3, 4]{Hale} and oscillatory solutions occur. If the feedback is of
constant sign, and the delay $(t-\tau (t))$ is bounded by a sufficiently
small constant, the infinite family of oscillatory solutions has amplitudes
that tend to zero and are bounded by the family of positive solutions
(finite yet no longer one dimensional, see \cite[p. 203]{myshkis72}).
Increasing the delay allows for a finite number of unbounded solutions to
exist within the oscillatory family, while the rest tend to zero.

Intriguingly, the transition between bounded and unbounded oscillatory
solutions corresponds to a critical case of periodic oscillations, which are
unique for all nonautonomous equations. What is more, the critical delay for feedback of fixed sign, is in fact a
threshold of oscillation frequency (semicycle length), which may be extended
to \textquotedblleft mixed feedback\textquotedblright
(where $p$ is not of constant sign). Let us now describe the first, and foremost, results on
the asymptotic behavior of \eqref{06.18}, and their relation to the
oscillation speed of the solutions. Most celebrated among such results is
perhaps the $\frac{3}{2}-$criterion below.

\begin{proposition}
When $p(t)\leq 0,t\in 
\mathbb{R}
$, all oscillatory solutions of \eqref{06.18} are bounded when 
\begin{equation}
\sup_{t\geq t_{0}}\int_{\tau (t)}^{t}|p(w)|dw\leq \frac{3}{2}  
\label{3_by_22}
\end{equation}%
and tend to zero at infinity when the strict inequality holds in \eqref{3_by_22}.
\end{proposition}

This result was first proven by Myshkis \cite{myshkisbook51,myshkispaper51,myshkis72}, under the assumption that $\inf_{t\geq
t_{0}}|p(t)|>0$ and $\sup_{t\in {{{{{{{\mathbb{[}}}}}}}}t_{0},+\infty
)}(t-\tau (t))~\cdot \sup_{t\in {{{{{{{\mathbb{[}}}}}}}}t_{0},+\infty
)}|p(t)|<\frac{3}{2}$. It was subsequently generalized to \eqref{3_by_22} in several papers
(see \cite{buch, lillo, 1, yoneyama, yorke}),
and even to a nonlinear equation in \cite{yorke}.

This $\frac{3}{2}-$criterion follows from bounds on the maximal intervals
which allow for increasing positive solutions, imposed by the sign of the
coefficient. These growth bounds are explicitly and intrinsically related
to the following antiperiodic solution $\varpi ^{-}$ of \eqref{06.18},
satisfying $p(t)\equiv -1,\tau _{\max }:=\sup_{t\geq t_{0}}(t-\tau (t))=%
\frac{3}{2}$, which first appeared in \cite{myshkisbook51, myshkispaper51}, and to which oscillatory solutions of \eqref {06.18} are
compared over a semicycle:
\begin{equation}
\begin{array}{lll}
\varpi ^{-}(t)=1-t, & \tau (t)=0, & t\in [ 0,3/2] \\ 
\varpi ^{-}(t)=-1/2-\int_{0}^{t-3/2}(1-u)du, & \tau (t)=t-3/2, & t\in
[ 3/2,5/2] \\ 
\varpi ^{-}(t)=-\varpi ^{-}(t+5/2), & \tau (t+5/2)=\tau (t)+5/2, & t\in {{{%
\mathbb{R}}}}\text{.}%
\end{array}
\label{11}
\end{equation}

The reader should note that we may (and henceforth shall) speak of $\tau
_{m}:=\sup_{t\geq t_{0}}(t-\tau (t))$ instead of $\sup_{t\geq
t_{0}}\int_{\tau (t)}^{t}|p(w)|dw$, and assume $|p(t)|\equiv 1$, as they are
equivalent under a change of variables (first introduced by Ladas et al. 
\cite{ladas}, see also {\ref{appendixB}}).

\bigskip For  positive feedback , $%
p(t)\geq 0,t\in 
\mathbb{R}
$, we have a similar result.

\begin{proposition}
\label{Proposition11_09}
\bigskip When $p(t)\geq 0,t\in 
\mathbb{R}
$, all oscillatory solutions of \eqref{06.18} are bounded when 
\begin{equation}
\sup_{t\geq t_{0}}\int_{\tau (t)}^{t}|p(w)|dw\leq 2+\frac{3}{4}+\ln 2
\label{5_29}
\end{equation}%
and tend to zero at infinity when the strict inequality holds in \eqref{5_29}.
\end{proposition}

Proposition~{\ref{Proposition11_09}} 
was conjectured (see \cite{myshkisgerman}, \cite[p. 171]%
{myshkis72}) since the discovery in the MSc thesis of E.\,I.~Soboleva in 1953
of an antiperiodic solution $\varpi ^{+}$ of \eqref{06.18} satisfying $%
p(t)\equiv +1,\tau _{m }=2+\frac{3}{4}+\ln 2$ :%
\begin{eqnarray}
\varpi ^{+}(t) &=&\left\{ 
\begin{array}{ll}
1-t, & t\in [ 0,9/8] \\ 
-1/8-\int_{0}^{t-9/8}(1-u)du, & t\in [ 9/8,13/8] \\ 
-1/2\exp\left\{  t-13/8\right\} , & t\in [ 13/8,13/8+\ln 2] \\ 
-\varpi ^{+}(t+\ln 2+13/8), & t\in {{{\mathbb{R}}}}\text{.}%
\end{array}%
\right.   \label{17.07} \\
&&%
\begin{array}{ll}
\tau (t)=-\ln 2-13/8, & t\in [ 0,9/8] \\ 
\tau (t)=-\ln 2-13/8+(t-9/8), & t\in [ 9/8,13/8] \\ 
\tau (t)=t, & t\in [ 13/8,13/8+\ln 2] \\ 
\tau (t+\ln 2+13/8)=\tau (t)+\ln 2+13/8, & t\in {{{\mathbb{R}}}}\text{.}%
\end{array}
\notag
\end{eqnarray}%
Proposition {\ref{Proposition11_09}} was proven by Lillo \cite{lillo} (for
piecewise continuous parameters), building upon the previous work of
Buchanan \cite{buch, buch1971,buch1974}.

\bigskip In the case of \textquotedblleft mixed feedback\textquotedblright
, where $p$ is not of constant sign, the delay cannot directly determine the
growth of oscillatory solutions on its own. However, one does obtain results
similar to those of the negative and positive feedback cases, introducing
the semicycle length as an additional parameter. For results obtained by
different methods, related to admissibility, see \cite{azb,bb7,bb9,bb11,bb10, gyor.hart., ght}.

In the case of $|p(t)|\equiv 1$, it was shown by Stavroulakis and Braverman 
\cite{1} that all oscillatory solutions with semicycles of maximum length $%
\ell \leq 2$ are bounded, and under the strict inequality $\ell <2$, they
tend to zero (cf. \cite[Theorem 12]{myshkisbook51}, \cite[Theorem 12]%
{myshkispaper51}). Moreover, this bound can be improved under restrictions
on the maximum delay, obtaining a spectrum of criteria that ensure that
oscillatory solutions tend to zero.

\begin{proposition}[\cite{1}]
\bigskip Assume $|p(t)|\equiv 1$ and set 
\begin{equation}
\Lambda (s):=\left\{  \begin{array}{ll} 2+s-\sqrt{2s}-\ln \left( \sqrt{2s}-1\right) , & s\in
[ 1,2] \\ 2, & s\geq 2.\end{array} \right.  \label{06.09}
\end{equation}%
If the semicycle length of an oscillatory solution of \eqref{06.18} is
bounded above by $\Lambda (\tau _{m })$, then the solution is bounded. If
the semicycle length is bounded above by a strictly smaller constant, the
solution tends to zero at infinity.
\end{proposition}

\begin{definition}
We will call a solution of \eqref {06.18}, $\alpha$-rapidly oscillating or
simply $\alpha -$oscillating, if eventually the length of its
nonoscillation intervals is bounded by $\alpha$, where $\alpha \in
(0,+\infty )$.
\end{definition}

As in the previous cases, the following threshold nonnegative $\Lambda (\tau
_{m })-$periodic solutions of \eqref{06.18}, to which the absolute value
of arbitrary oscillatory solutions was compared for each $\tau _{m}\in
[ 1,2]$, were instrumental in the methods of \ \cite{1}:%
\begin{eqnarray}
~~\varpi _{\tau _{m }}(t) &=\left\{ 
\begin{array}{l}
t, ~ t\in [ 0,\tau _{m }-1] \\ 
\tau _{m }-1-\int_{0}^{t-\left( \tau _{m
}-1\right) }(1-u)du, ~ t\in [ \tau _{m }-1,\left( \tau _{m
}+1\right) -\sqrt{2\tau _{m }}] \\ 
\exp (t-  \tau _{m}-1 +%
\sqrt{2\tau _{m }} )\left( \sqrt{2\tau _{m }}-1\right) , \\  t\in
[ \tau _{m}+1-\sqrt{2\tau _{m}},\Lambda (\tau
_{m})-1] \\ 
1-(t-\left( \Lambda (\tau _{m})-1\right) ), ~ 
t\in {{{{{{{\mathbb{[}}}}}}}}\Lambda (\tau _{m})-1,\Lambda (\tau _{m
})] \\ 
\varpi _{\tau _{m }}(\Lambda (\tau _{m
})+t), ~ t\in {{{{{{{\mathbb{R.}}}}}}}}%
\end{array}%
\right.   \label{14} \\
&%
\begin{array}{ll}
\tau (t)=-1, & t\in [ 0,\tau _{m }-1] \\ 
\tau (t)=-1+(t-\tau _{m}-1), & t\in [ \tau _{m }-1,\left( \tau
_{m}+1\right) -\sqrt{2\tau _{m }}] \\ 
\tau (t)=t, & t\in [ \left( \tau _{m}+1\right) -\sqrt{2\tau _{m
}},\Lambda (\tau _{m })-1] \\ 
\tau (t)=\Lambda (\tau _{m })-1, & t\in {{{{{{{\mathbb{[}}}}}}}}\Lambda
(\tau _{m })-1,\Lambda (\tau _{m })] \\ 
\tau (t+\Lambda (\tau _{m }))=\tau (t)+\Lambda (\tau _{m }), & t\in {{{%
{{{{\mathbb{R.}}}}}}}}%
\end{array}
\notag
\end{eqnarray}%
Furthermore, to each such solution $\varpi _{\tau _{m}}$, there
corresponds an $\Lambda (\tau _{m })-$antiperiodic solution $\overset{\sim }{\varpi} _{\tau
_{m}}$ of \eqref {06.18}, with positive feedback $p(t)\equiv +1$ and $%
\overset{\sim }{\tau }_{m }:=2+2\tau _{m }-\sqrt{2\tau _{m}}-\ln
\left( \sqrt{2\tau _{m }}-1\right) $:

\begin{equation*}
\begin{array}{l}
\overset{\sim }{\varpi} _{\tau _{m }}(t)=\varpi _{\tau _{m }}(t),~t\in [ 0,2+\tau
_{m}-\sqrt{2\tau _{m }}-\ln \left( \sqrt{2\tau _{m}}-1\right) ]
\\ 
\overset{\sim }{\varpi} _{\tau _{m}}(t)=-\overset{\sim }{\varpi} _{\tau _{m }}(t+2+\tau _{m }-\sqrt{2\tau
_{m}}-\ln \left( \sqrt{2\tau _{m }}-1\right) ),~t\in {{{{{{{\mathbb{R}%
}}}}}}}%
\end{array}%
\end{equation*}%
Strikingly, $\varpi ^{+}$ is equal to $\overset{\sim }{\varpi} _{\frac{9}{8}}$ up to
translation, which is the $\overset{\sim }{\varpi} $ corresponding to the least maximum delay,
indicating a deep connection between oscillatory solutions of the positive
feedback and the mixed feedback equations. Namely, it would seem that the
sole \textquotedblleft cause\textquotedblright\ of unboundedness in the
positive feedback case is that the function may be rewritten as a unbounded solution of an equation with mixed feedback.

Perhaps the most interesting property of the special periodic
solutions is uniqueness and uniformity: nontrivial oscillatory solutions are
asymptotic to the special periodic solutions, when the delay is critical.
This was proven by Lillo \cite[Theorem 3.1, Theorem 3.3]{lillo} for the
negative feedback, and both Lillo and Buchanan further conjectured that
oscillatory solutions of \eqref{06.18} with $p(t)\geq 0,t\in 
\mathbb{R}
$, are asymptotic to a multiple and translation of $\varpi ^{+}$, proving a
partial result in this direction \cite{buch}, \cite[Theorem 3]{buch1971}.
For more relevant results proven by these techniques, the reader is referred
to \cite{buch1971,buch1974}. More recently, Stavroulakis and Braverman \cite%
{1}, inspired by the remarks of Buchanan \cite[p. 52]{buch}, \cite[p. 676]%
{buch1971}, and Lillo \cite[p. 13]{lillo}, further conjectured a stronger
notion of convergence to the periodic solutions in the critical case. They
furthermore set the computation of the critical semicycle length $\Lambda
(\tau )$ for $\tau \in (\frac{1}{e},1)$ as an open problem.

In the present paper, we extend the definition of $\Lambda (\tau ),\varpi _{\tau
}, $ to $\tau \in (\frac{1}{e},2],$ and calculate the asymptotics of $\Lambda
(\tau )$ near $\frac{1}{e}$ (Section 4, {\ref{appendixA}}), using certain auxiliary results (Section 3). We confirm the
conjecture of \cite{1} for $\tau _{m}\in (\frac{1}{e},2)$ and amend it for $%
\tau _{m}=2$ (Section 5)$.$ In the critical case of mixed feedback, where $%
\tau _{m}=\tau \in (\frac{1}{e},2)$, any $\Lambda (\tau )-$oscillating
solution will be asymptotic to $\varpi _{\tau }$ up to
scaling and translation, differing by $o(1)$ (we use the standard Landau
notation throughout). Reasoning by analogy and considering that $\varpi ^{+}$
is equal to $\overset{\sim }{\varpi }_{\frac{9}{8}}$ up to translation, we
refine the conjecture of Buchanan \cite[p. 52]{buch}, \cite[p. 676]{buch1971}%
, and Lillo \cite[p. 3, 13]{lillo} that oscillatory solutions of the positive
feedback equation at the critical state would differ from a translation
of $\varpi ^{+}$ by $o(1)$ (Section 2). We finally discuss possible ways of tackling the
Buchanan-Lillo conjecture, and why it remains open (Section 6).

\section{Rates of convergence}

The following asymptotic result was proven by Lillo \cite{lillo}.

\begin{proposition}[{\protect\cite[Theorem 3.1, Theorem 3.3]{lillo}}]
\label{Proposition11.24}Assume that $x$ solves 
\begin{equation*}
x^{\prime }(t)=-x(\tau (t)),~~t\geq t_{0}
\end{equation*}%
with $t-\frac{3}{2}\leq \tau (t)\leq t,\forall t\geq 0$. Then there exists a
sequence $t_{i}\ $such that $\lim_{i\rightarrow \infty }(t_{i+1}-t_{i})=%
\frac{5}{2}$ and 
\begin{equation}
\lim_{t\rightarrow \infty }(x(t)-y(t))=0  \notag
\end{equation}%
where 
$$
y(t) :=M(-1)^{i}\varpi ^{-}(t-t_{i}),t\in [ t_{i},t_{i+1}], \quad  
M :=\limsup_{t\rightarrow \infty }|x(t)| 
$$
and $\varpi ^{-}$ is given in \eqref{11}.
\end{proposition}

\begin{remark}
Proposition {\ref{Proposition11.24}} cannot be strengthened to%
\begin{equation}
x(t)=\left( \limsup_{t\rightarrow \infty }|x(t)|\right) \varpi
^{-}(t+const)+o(1),  \notag
\end{equation}%
as the following Example illustrates.
\end{remark}

\begin{example}
\label{Example11.29}
We set $x(t)=\varpi ^{-}(t),t\in (-\infty ,\frac{3}{2}]$
and
\begin{align*}
 x^{\prime }(t)  & =  x\left( m\frac{5}{2}-\sum_{i=1}^{m}\frac{1}{i}\right)  
\left( \varpi^{-} \right)^{\prime }\left( t-\left[ m\frac{5}{2}-\sum_{i=1}^{m}\frac{1}{i}\right]
\right) , \\
t  & \in  \left[ m\frac{5}{2}-\sum_{i=1}^{m}\frac{1}{i},(m+1)\frac{5}{2} -\sum_{i=1}^{m+1}\frac{1}{i}\right],
\end{align*}%
where $m \in {\mathbb N}$. Obviously,%
\begin{eqnarray*}
x(t) & =x\left( m\frac{5}{2}-\sum_{i=1}^{m}\frac{1}{i}\right) \varpi ^{-}\left(
t-\left[ m\frac{5}{2}-\sum_{i=1}^{m}\frac{1}{i}\right] \right) , \\  t  & \in \left[
m\frac{5}{2}-\sum_{i=1}^{m}\frac{1}{i},(m+1)\frac{5}{2}-\sum_{i=1}^{m+1}%
\frac{1}{i}\right],  \\
& x\left( m\frac{5}{2}-\sum_{i=1}^{m}\frac{1}{i} \right)=(-1)^{m}\prod_{i=1}^{m} \left( 1-\frac{1}{%
2i^{2}} \right). &
\end{eqnarray*}%
Notice that $\sum\limits_{i=1}^{\infty }\frac{1}{i}$ diverges but $%
\prod\limits_{i=1}^{\infty }(1-\frac{1}{2i^{2}})$ does converge to a positive
constant (because the sum of logarithms converges). Considering the sequence
of zeros of $x$ and of shifts of $\varpi ^{-}$, using \eqref {11}, we may
exclude $x(t)=\left( \limsup_{t\rightarrow \infty }|x(t)|\right) \varpi
^{-}(t+const)+o(1)$.
\end{example}

\bigskip The failure of the stronger convergence illustrated in Example {\ref%
{Example11.29} is due to the the derivative of }$\varpi ^{-}$ tending to zero
near the extrema of $\varpi ^{-}.$ In fact, the expansion of the derivative
near its zeros is closely related to the rate of convergence. If, on the
other hand, the derivative is separated from zero, oscillatory solutions
tend to a shift of the periodic solution. Such is the case for $\varpi
_{\tau },$ which satisfies 
\begin{equation*}
{\rm essinf}_{t\in 
\mathbb{R}
}|\varpi _{\tau }^{\prime }(t)|>0.
\end{equation*}%
Notice that because the parameters of the equation are only piecewise
continuous, the solution is only absolutely continuous and not continuously
differentiable. Hence its derivative may be separated from zero, despite the
extrema.

\bigskip The main result of this paper characterizes the rate of convergence
to $\varpi _{\tau }$ at the critical state, while extending the definition
of $\Lambda (\tau ),\varpi _{\tau },$ to $\tau \in (\frac{1}{e},2].$ For $%
\tau \in [ 1,2)$ we confirm the conjecture in 
\cite{1}:

{\bf Main Result}
\label{Proposition5.22}
 Let $x$ be a $\Lambda (\tau )-$oscillating
solution of \eqref {06.18} with $|p(t)|\equiv 1,\tau _{m}\leq \tau \in
[ 1,2)$. Then 
\begin{equation}
|x(t)|=M\varpi _{\tau }(t+\eta )+o(1)  \label{11.03}
\end{equation}%
where $\varpi _{\tau }$ is given in \eqref {14}, $M=\limsup_{t\rightarrow
\infty }|x(t)|$, and $\eta \in [ 0,\Lambda (\tau ))$.

Another factor which may affect the rate of convergence is the interposition
of intervals between the semicycles of the periodic solution. This is
possible when the derivative is piecewise constant, and equal to the maximum
possible in absolute value. Such is the case for $\tau _{m}\geq 2$. We will
see that an interval with length bounded from \ below, where the delayed
argument is in the descent of the previous semicycle, is crucial to Theorem {\ref{Theorem11.05}}.

\begin{example}
\label{Example11.28}We set $x(t)=\varpi _{2}(t),t\in (-\infty ,0]$ and for a
large integer $N\in 
\mathbb{N}
$, 
$$
t_{n+1} = t_{n}+2+\frac{1}{n+N}, ~n=0,1,2,..., \quad
t_{0} =-2
$$
and%
\begin{equation*}
\begin{array}{ll}
x^{\prime }(t)=0, & t\in [ t_{n}+2,t_{n+1}]    \\ 
x^{\prime }(t)=x(t_{n}+1), & t\in [ t_{n+1},t_{n+1}+(1-\frac{1}{n+N})]
  \\ 
x^{\prime }(t)=1-t+\left[ t_{n+1}+(1-\frac{1}{n+N})\right] , & t\in [
t_{n+1}+1-\frac{1}{n+N},t_{n+1}+1]   \\ 
x^{\prime }(t)=-x(t_{n+1}+1), & t\in [ t_{n+1}+1,t_{n+1}+2] 
\end{array}%
\end{equation*}%
Obviously, 
\begin{equation*}
x(t_{n+1}+1)=\prod\limits_{i=0}^{n}\left( 1-\frac{1}{2\left( i+N\right) ^{2}%
}\right) .
\end{equation*}%
However, $x(t)=\left( \limsup\limits_{t\rightarrow \infty }x(t)\right) \varpi
_{2}(t+const)+o(1)$ is impossible, as $\sum\limits_{i=N}^{\infty }\frac{1}{i}$
diverges but $\prod\limits_{i=N}^{\infty }(1-\frac{1}{2i^{2}})$ does converge to a
positive constant.
\end{example}

Considering the main result of this paper and the aforementioned
relationship between the positive and mixed feedback cases at the critical
state ($\varpi ^{+}$ is equal to $\overset{\sim }{\varpi} _{\frac{9}{8}}$ up
to translation), we refine the conjecture of Buchanan \cite[p. 52]{buch}, 
\cite[p. 676]{buch1971}, and Lillo \cite[p. 3, 13]{lillo}:
\vskip 1pt
\textquotedblleft it seems plausible that the only solutions which do not converge to $0$ are those that tend as a limit to a constant multiple of [$\varpi ^{+}$]\textquotedblright \cite[p. 52]{buch};
\vskip 1pt
\textquotedblleft The results obtained [...] suggest that all oscillatory
solutions for [$p(t)\equiv 1, \tau _{m}=2+\frac{3}{4}+\ln 2$...] which do not tend to
zero as $t \rightarrow \infty$  will be asymptotic to a scalar multiple of [$\varpi ^{+}$]  \textquotedblright \cite[p. 3]{lillo}.
\vskip 1pt
While Buchanan and
Lillo did not explicitly define what kind of convergence they felt was
plausible, the analogy of \eqref {11.03} motivates the following:

\bigskip
{\bf Buchanan-Lillo Conjecture}
Consider an oscillatory solution $x$ of \eqref {06.18} with 
\begin{equation*}
p(t)\equiv 1,\tau _{m}=2+\frac{3}{4}+\ln 2.
\end{equation*}
Then 
\begin{equation*}
x(t)=\left( \limsup_{t\rightarrow \infty }|x(t)|\right) \varpi
^{+}(t+const)+o(1),
\end{equation*}%
where $\varpi ^{+}$ is given in \eqref{17.07}.

\bigskip

\bigskip \bigskip The derivative of $\varpi ^{+}$ being separated from zero,
and the fact that the delay in the ascent of one semicycle is in the
descent of the previous semicycle (preventing one from interposing intervals
where the function is zero between semicycles), {indicate that
counterexamples such as Examples \ref{Example11.29}, \ref{Example11.28}, are
impossible. This speaks for the validity of the stronger\ convergence to the 
} $\varpi ^{+}${, in the sense of the conjecture. }

\section{Auxiliary results}

\bigskip The following Lemma enables us to restrict attention to nonnegative
solutions, equivalently, to the absolute value of solutions.

\begin{lemma}[{{\protect\cite[Theorem 1]{serrin}},{\protect\cite{tandra}}}]
\label{Lemma11.30}
When \bigskip $x$ solves%
\begin{equation*}
x^{\prime }(t)=p(t)x(\tau (t)),~~t\geq t_{0}
\end{equation*}%
then $|x|$ also solves%
\begin{eqnarray*}
|x|^{\prime }(t)=\widetilde{p}\left( t\right) |x|(\tau (t)), & t\geq t_{0}& \\
\widetilde{p}\left( t\right) :=\left\{ 
\begin{array}{cc}
sgn\left[ x(t)x(\tau (t))\right] p(t), & x(t)x(\tau (t))\neq 0, \\ 
p(t), & x(t)x(\tau (t))=0,%
\end{array}%
\right.  &&
\end{eqnarray*}%
where $sgn(\cdot)$ is the sign function.
Further, we have 
\begin{equation*}
|\widetilde{p}(t)|=|p(t)|,\text{a.e.  for }t\geq t_{0}.
\end{equation*}
\end{lemma}

The following function was first investigated by Myshkis \cite%
{myshkispaper51}, in relation to oscillation problems.
We will prove that it is an upper bound on the interval where oscillatory
solutions decrease from a maximum to the next zero.

\begin{lemma}[\protect{\cite{myshkispaper51}, \cite[Theorem 5]{2}}]
\label{Lemma3.35}For each fixed $\tau >\frac{1}{e}$, denote by $x_{\tau }$
the solution of 
\begin{equation}
x_{\tau }^{\prime }(t)+x_{\tau }(t-\tau )=0,~~t\geq 0, \quad
x_{\tau }(t)=1,~~t\leq 0.%
\label{03.55}
\end{equation}%
This function strictly decreases from one to its first zero 
\begin{equation}
\varrho(\tau):=\inf\{t>0:x_{\tau}(t)=0\}.
\label{08.06}
\end{equation}%
This first root $\varrho (\tau )$ is a
continuous function of $\tau $, strictly decreasing on $(\frac{1}{e},1]$, constant $\varrho(\tau)\equiv1$ for $\tau\geq1$,
and satisfies the following asymptotic expansion near $\frac{1}{e}^{+}:$
\begin{equation*}
\varrho (\tau )=\frac{\pi }{\sqrt{2e^{3}}}\left( \tau -\frac{1}{e}\right)
^{-1/2}+o\left( \left( \tau -\frac{1}{e}\right) ^{-1/2}\right) .
\end{equation*}

\end{lemma}

The next Lemma is a backwards analog and generalization of a well-known
comparison theorem (\cite[Theorem 1.1]{gyori:96}, \cite[Theorem 1]{kwong}).
As previously remarked, it suffices to study only nonnegative solutions.
Because the equations are moreover linear, it is sufficient - and simplest -
to consider solutions with values between $0$ and $1$.

\begin{lemma}
\label{Lemma7.39}
Assume for $\tau >\frac{1}{e}$ 
\begin{eqnarray}
y^{\prime }(t)+c(t)y(t-\sigma (t))=0,t\in [ 0,\varrho (\tau )]
\label{19.20} \\
1\geq \max_{[ -\tau ,\varrho (\tau )]}y>0, ~~\min_{[-\tau
,\varrho (\tau )]}y\geq 0, \quad
y(\varrho (\tau ))=0,  \notag
\end{eqnarray}%
where the functions $\sigma,c$ are measurable, $0\leq \sigma (t)\leq \tau $ and $|c(t)|\leq 1$, and $\varrho $
is the first root of the solution $x_{\tau }$ of \eqref {03.55}  described in Lemma~{\ref{Lemma3.35}}. Then, 
\vskip 1pt
I)
\begin{equation}
y(t) \leq x_{\tau }(t),t\in [ 0,\varrho (\tau )].  \label{9.10}
\end{equation}%
\vskip 1pt
II)\begin{equation*}
a(t):=x_{\tau }(t)-y(t)\geq 0,t\in [ 0,\varrho (\tau )]
\end{equation*}
is a nonincreasing
function.
\vskip 1pt
III) If, for a fixed $\tau \in (\frac{1}{e},1)$ and
sufficiently small $\delta \in (0,1)$, we have $a(\varrho
(\tau )-\tau)\leq \delta$ then 
\begin{equation}
0\leq x_{\tau }(t)-y(t)=a(t)\leq O\left( \delta ^{2^{-\left\lfloor \frac{\varrho
(\tau )}{\tau }\right\rfloor }}\right) ,t\in [ 0,\varrho (\tau )-\tau ],%
\text{ }  \label{8.37}
\end{equation}%
where $\left\lfloor \cdot \right\rfloor $ is the greatest integer function
and $O\left( \delta ^{2^{-\left\lfloor \frac{\varrho (\tau )}{\tau }%
\right\rfloor }}\right)$ is a bound which is uniform in $t\in [ 0,\varrho (\tau )-\tau ]$, it  varies with respect to $\delta\rightarrow 0^{+}$ 
according to the standard Landau notation, and depends only on the fixed $\tau \in (\frac{1}{e},1)$.
\end{lemma}

\begin{proof}
I)
Consider the sequence of functions 
\begin{eqnarray}
\varphi _{0}(t) \equiv 1, \quad
\varphi _{n+1}(t)  = \min \left\{\int_{t}^{\varrho (\tau )}\varphi _{n}(s-\tau
)ds,1\right\},t\leq \varrho (\tau ). \label{09.40}
\end{eqnarray}%
We show that the function sequence $\varphi _{n}$ is pointwise nonincreasing, and each function is nonnegative. The inequality
\begin{equation}
0\leq\varphi _{n+1}(t)\leq \varphi _{n}(t),~~t\leq \varrho(\tau)
\label{9.13}
\end{equation}%
holds by definition for $n=0$. Integrating \eqref{9.13} from $t$ to $\varrho(\tau)$ for a given nonnegative integer $n$, and using definition \eqref{09.40}, we immediately see that \eqref{9.13} holds for $n+1$.
\vskip 1pt
We furthermore notice that for $n=0$
\begin{equation}
x_{\tau}(t)\leq \varphi _{n}(t), ~~t\leq \varrho(\tau).
\label{9.17}
\end{equation}%
Assuming \eqref{9.17} for a given nonnegative integer $n$ we will show that \eqref{9.17} holds for $n+1$. Firstly, integrating \eqref{9.17} and \eqref{03.55}, using definition \eqref{09.40}, we obtain
\begin{equation}
\varphi _{n+1}(t)=\min \left\{\int_{t}^{\varrho (\tau )}\varphi _{n}(s-\tau
)ds,1\right\}\geq\min \left\{\int_{t}^{\varrho (\tau )}x _{\tau}(s-\tau
)ds,1\right\}=1, ~t\leq0.
\label{05.37}
\end{equation}
Now, integrating \eqref{03.55}, using \eqref{09.40}, \eqref{9.17},
\begin{eqnarray*}%
x_{\tau}(t) \leq \min \left\{ \int_{t}^{\varrho (\tau )}x_{\tau}(s-\tau
)ds,1\right\} \leq \varphi _{n+1}(t), ~t\in [0,\varrho(\tau)].
\end{eqnarray*}
We conclude that \eqref{9.17} holds for all positive integers $n$.
Using the definition of $\varphi _{n}(t)$ in \eqref{09.40} and \eqref{9.17}, the monotone sequence of functions $\varphi _{n}(t)$ must tend to the solution $x_{\tau }(t)$ of \eqref{03.55}, by monotone convergence. 
\vskip 1pt
We will now similarly show that 
\begin{equation}
y(t)\leq \varphi _{n}(t),~~t\leq \varrho(\tau)
\label{9.18}
\end{equation}%
holds for all nonnegative $n$. It holds for $n=0$ by definition. Assuming \eqref{9.18} for a given nonnegative integer $n$ we will show that \eqref{9.18} holds for $n+1$.
Integrating \eqref{19.20}, using \eqref{09.40}, \eqref{9.18}, and the fact that $\varphi _{n}$ is nonincreasing with respect to time,
\begin{align*}%
y(t) & \leq \min  \left\{  \int_{t}^{\varrho (\tau )}c(s)y(s-\sigma(s)
)ds,1\right\} \leq \min  \left\{ \int_{t}^{\varrho (\tau )}y(s-\sigma(s)
)ds,1\right\}  \\   & \leq \min  \left\{ \int_{t}^{\varrho (\tau )}\varphi _{n}(s-\sigma(s)
)ds,1 \right\} \leq \varphi _{n+1}(t),
\end{align*}
for $t\in [0,\varrho(\tau)]$. This, together with $1\geq \max_{[ -\tau ,\varrho (\tau )]}y$ and \eqref{05.37}, shows that \eqref{9.18} holds for all positive $n$. This, together with the fact that $\varphi_{n}$ tends to $x_{\tau}$, proves \eqref {9.10}.
\vskip 1pt
II) We estimate the derivative of $a(s), s\in [0, \varrho(\tau)]$, using the fact that $x_{\tau }$ is nonincreasing and $\sigma(t)\leq \tau$, 
inequality \eqref {9.10}:
\begin{eqnarray*}
a'(s) &= &-x_{\tau}(s-\tau)+c(s)y(s-\sigma(s))\leq-x_{\tau}(s-\tau)+y(s-\sigma(s)) \\ &  \leq & -x_{\tau}(s-\sigma(s))+y(s-\sigma(s)) \leq 0.
\end{eqnarray*}
Together with \eqref{9.10}, this proves the assertion.
\vskip 1pt
III)
Let us assume $\tau \in (\frac{1}{e},1)$ and $a(\varrho (\tau )-k\tau )\leq \delta,$ where for the integer $k$ we have $1\leq k\leq \left\lfloor 
\frac{\varrho (\tau )}{\tau }\right\rfloor $. By assumption, this holds for $k=1$. We will prove the desired result by induction on $k$. Using $a(t)\geq 0$, which we have shown in II),
\begin{equation*}
\int_{\varrho (\tau )-k\tau }^{\varrho (\tau )-(k-1)\tau}
a'(w)dw=a(\varrho (\tau )-(k-1)\tau )-a(\varrho (\tau )-k\tau )\geq -\delta.
\end{equation*}%
Because we have shown in II) that $a'(s)\leq 0$, denoting by $m(S)$ the Lebesgue measure of any measurable set $S\subset 
\mathbb{R}
$, this implies 
\begin{equation}
\begin{array}{ll}
 \displaystyle  m  \left\{  t\in [ \varrho (\tau )-k\tau ,\varrho (\tau )-(k-1)\tau ]:  \right.  & a'
(t)=-x_{\tau}(t-\tau) \\ & \displaystyle +  \left.  c(t)y(t-\sigma(t))\leq -\sqrt{\delta }  \right\}  \leq \sqrt{\delta }. 
\end{array}
\label{5.99}
\end{equation}%
We consider three cases.
\vskip 1pt
$i$)
 $\varrho (\tau )-(k+1)\tau \geq 0$.

By \eqref{5.99}, for any small $\varepsilon >0,\exists $ $ \displaystyle \eta \in \left[
\varrho (\tau )-k\tau ,\varrho (\tau )-k\tau +\sqrt{\delta }+\varepsilon  \right]$
 such that 
\begin{equation}
 y \left( \eta -\sigma (\eta ) \right)  \geq x_{\tau
}(\eta -\tau )- \sqrt{\delta }.  \label{1.34}
\end{equation}%
We also have 
\begin{equation}
\begin{array}{ll}
y \left( \eta -\sigma (\eta ) \right)  &  \displaystyle \geq x_{\tau } \left( \eta -\tau +\frac{\sqrt{\delta }}{x_{\tau
}(\varrho (\tau )-\tau )} \right) \\ &  \displaystyle \geq x_{\tau } \left( \varrho (\tau
)-k\tau +\sqrt{\delta }+\varepsilon -\tau +\frac{\sqrt{\delta } }{x_{\tau}(\varrho
(\tau )-\tau )} \right),  \end{array} 
\label{4.03}
\end{equation}%
where the first inequality follows from \eqref{1.34} and $-1 \leq x'_{\tau}(t)\leq -x_{\tau}(\varrho(\tau )-\tau ),t\in[0,\varrho(\tau)]$, 
the second inequality by considering the interval where $\eta$ is assumed to exist. By \eqref {9.10}, the first inequality in \eqref{4.03}, and the strictly decreasing nature of $x_{\tau}$ on $[0,\varrho(\tau)]$,
\begin{equation}
\sigma (\eta )\geq \tau -\frac{\sqrt{\delta }}{x_{\tau }(\varrho (\tau )-\tau )}. \label{06.20}
\end{equation}%
Considering the interval where $\eta$ is assumed to exist, using \eqref{06.20} and the triangle inequality, we obtain
\begin{equation}
|\varrho (\tau )-(k+1)\tau-(\eta -\sigma (\eta ))|\leq 
|\varrho (\tau )-k\tau-\eta|+|\tau -\sigma (\eta )|\leq \sqrt{\delta }%
+\varepsilon +\frac{\sqrt{\delta } }{x_{\tau }(\varrho (\tau )-\tau )}.
\label{6.22}
\end{equation}
Moreover, $x_{\tau },y$ are Lipschitz with uniform Lipschitz constant $1$ on $%
[0,\varrho (\tau )]$. We therefore have 
\begin{eqnarray*}
y(\varrho (\tau )-(k+1)\tau ) &\geq &y(\eta -\sigma (\eta ))-\sqrt{\delta }%
-\varepsilon -\frac{\sqrt{\delta } }{x_{\tau }(\varrho (\tau )-\tau )}
\\
&\geq &x_{\tau }(\varrho (\tau )-(k+1)\tau )-2 \left( \sqrt{\delta }%
+\varepsilon +\frac{\sqrt{\delta }}{x_{\tau }(\varrho (\tau )-\tau )}  \right),
\end{eqnarray*}%
where the first inequality follows from \eqref{6.22} and the Lipschitzian nature of $y$, and the second from \eqref{4.03} and the Lipschitzian nature of $x_{\tau}$. We have proved that 
\begin{equation*}
a(\varrho (\tau )-(k+1)\tau)\leq O(\sqrt{\delta}).
\end{equation*}
This concludes the induction step.
\vskip 1pt
$ii$)
 $\varrho (\tau )-(k+1)\tau <0$, $\varrho (\tau )-k\tau >\sqrt{\delta }$.
 \vskip 1pt
We may proceed similarly to the above and, for any sufficiently small $\varepsilon >0$, consider an $\eta $ $\in
[ \tau ,\tau +\sqrt{\delta }+\varepsilon ]$ such that \eqref {1.34}, %
and the first inequality in \eqref {4.03} hold. This again gives \eqref{06.20}, and considering the interval where $\eta$ is now assumed to exist,
\begin{equation}
0\leq \eta -\sigma (\eta )\leq \sqrt{%
\delta }+\varepsilon +\frac{\sqrt{\delta } }{x_{\tau }(\varrho (\tau )-\tau )}. \label{8.47}
\end{equation}
The result follows from \eqref{8.47} and the Lipschitzian nature of $y,x_{\tau}$.
\vskip 1pt
$iii$) $\varrho (\tau )-(k+1)\tau <0,\varrho (\tau )-k\tau \leq \sqrt{\delta }%
.$ 
\vskip 1pt
The result follows immediately from the assumptions and the Lipschitzian nature of $y,x_{\tau}$.

By induction and the monotonicity of $a(t)$, we have the desired result.
\end{proof}
We now prove a strict version of \eqref{9.10}, when comparing $x_{\tau _{1}}$ and  $x_{\tau _{2}}$, where $\frac{1}{e}<\tau _{1}<\tau _{2}\leq 1$. 

\begin{corollary}
\label{Corollary3.19}
If $\frac{1}{e}<\tau _{1}<\tau _{2}\leq 1$ then $%
\varrho (\tau _{1})>\varrho (\tau _{2})$ and $x_{\tau _{1}}(\varrho (\tau
_{1})-t)<x_{\tau _{2}}(\varrho (\tau _{2})-t)$, $t\in (0,\varrho (\tau _{2})]$.
\end{corollary}

\begin{proof}
The fact that $\varrho (\tau _{1})>\varrho (\tau _{2})$ follows from Lemma {\ref{Lemma3.35}}
and $\tau _{1}<\tau _{2}\leq 1$. We also remark that $\varrho (\tau _{2})\geq\tau _{2}$. Applying Lemma {\ref{Lemma7.39}} with $y=x_{\tau _{1}}, \tau=\tau_{2}$, we have 
\begin{equation}
x_{\tau _{1}}(\varrho (\tau
_{1})-\tau_{2})\leq x_{\tau _{2}}(\varrho (\tau _{2})-\tau_{2}).
\label{9.34}
\end{equation}%
If we had equality in \eqref{9.34}, then we may let $\delta\rightarrow 0^{+}$ in relation \eqref {8.37}, contradicting $\varrho (\tau _{1})>\varrho (\tau _{2})$. This proves the strict inequality in \eqref{9.34}. Considering the monotonicity of $a$ proven in Lemma {\ref{Lemma7.39}}, the strict inequality in \eqref{9.34} gives
\begin{equation*}
x_{\tau _{1}}(\varrho (\tau
_{1})-t)<x_{\tau _{2}}(\varrho (\tau _{2})-t),t\in[\tau _{2},\varrho(\tau_{2})].
\end{equation*}%
By continuity, and Lemma {\ref{Lemma7.39}}, for a sufficiently small $\varepsilon>0$,
\begin{eqnarray}
x_{\tau _{1}}(\varrho (\tau
_{1})-t)<x_{\tau _{2}}(\varrho (\tau _{2})-t),~t\in[\tau _{2},\varrho(\tau_{2})+\varepsilon], 
\label{10.18}\\
x_{\tau _{1}}(\varrho (\tau
_{1})-t)\leq x_{\tau _{2}}(\varrho (\tau _{2})-t),~t>\varrho(\tau_{2})+\varepsilon.
\label{10.19}
\end{eqnarray}
By \eqref {03.55} and inequalities \eqref{10.18},\eqref{10.19}, we have
\begin{eqnarray*}
x'_{\tau _{1}}(\varrho (\tau
_{1})-t)>x'_{\tau _{2}}(\varrho (\tau _{2})-t),~t\in (0,\varepsilon)
\notag \\
x'_{\tau _{1}}(\varrho (\tau
_{1})-t)\geq x'_{\tau _{2}}(\varrho (\tau _{2})-t),~t\in (\varepsilon,\varrho (\tau _{2})].
\end{eqnarray*}%
Integrating, we obtain the desired result.
\end{proof}

\section{Critical semicycle length $\Lambda (\tau )$, $\tau \in (\frac{1}{e},1)$}

\subsection{Growth bounds within a semicycle}

The following function describes the maximal growth of solutions between a
zero and its next extremum, and helps us define $%
\Lambda (\tau ),\tau \in (\frac{1}{e},1).$

\begin{lemma}
\label{Lemma4.31}
For each fixed $\tau \in (\frac{1}{e},2]$, the solution $%
\psi _{\tau }$ of the differential equation 
\begin{equation}
\psi _{\tau }^{\prime }(t) =\left\{  \begin{array}{ll} \max \{x_{\tau }(\varrho (\tau
)+t-\tau ),\psi _{\tau }(t)\}, & t\in [ 0,\tau ] \\ \psi _{\tau }(t), & t\in
[ \tau ,\infty ) \end{array} \right. \quad \psi _{\tau }(0)=0,   \label{3.46} \\
\end{equation}
where $x_{\tau }$ is the solution of \eqref {03.55}{\ } and $\varrho $ is
its first root, described in Lemma {\ref{Lemma3.35}}, is strictly increasing
and satisfies $\psi _{\tau }$ $(\infty )=\infty ,$ and $\psi _{\tau }(\tau
)<1,\tau \in (\frac{1}{e},1]$. The strictly increasing, continuous function $%
\Theta (t):=\int_{0}^{t}x_{\tau }(\varrho (\tau )+w-\tau )dw-x_{\tau
}(\varrho (\tau )+t-\tau ),t\in [ 0,\tau ]$ satisfies $\Theta (0)$ $%
\Theta (\tau )<0$ and its unique root $\xi _{\tau }\in (0,\tau )$ is such
that 
\begin{eqnarray}
\int_{0}^{\xi _{\tau }}x_{\tau }(\varrho (\tau )+w-\tau )dw &=&x_{\tau
}(\varrho (\tau )+\xi _{\tau }-\tau ),  \label{4.32} \\
\psi _{\tau }^{\prime }(t) &=&x_{\tau }(\varrho (\tau )+t-\tau ),~t\in
[ 0,\xi _{\tau }],  \notag \\
\psi _{\tau }^{\prime }(t) &=&\psi _{\tau }(t),~t\in [ \xi _{\tau
},\infty ).  \notag
\end{eqnarray}
\end{lemma}

The following definition is a direct extension of \eqref{06.09}, \cite[Lemma 3.5]{1}, with
which it coincides for $\tau \in [ 1,2]$.

\begin{definition}
\label{Definition4.40}
For each fixed $\displaystyle \tau \in \left( \frac{1}{e},2\right]$, we define $%
\Lambda (\tau )>\varrho (\tau )$, where $\varrho $ is given in Lemma {\ref%
{Lemma3.35}}, by the threshold condition 
\begin{equation*}
\psi _{\tau }(\Lambda (\tau )-\varrho (\tau ))=1,
\end{equation*}%
where $\psi _{\tau }$ is the solution of \eqref {3.46}.
\end{definition}

\begin{remark}
\label{Remark10.36}
As direct consequences of Definition {\ref{Definition4.40}}, \eqref{06.09} and  \eqref{4.32},  we get
\begin{eqnarray}
\Lambda (\tau )>\varrho (\tau )+\tau ,\tau \in (\frac{1}{e},1], \notag\\
\xi_{\tau}>\tau-1,\tau\in[1,2).\label{10.39}
\end{eqnarray}
Inequality \eqref{10.39} guarantees there exists an interval where the delayed argument of the special solution is in the descent of the previous semicycle.
\end{remark}

Recalling a well-known comparison Theorem \cite{laks2, laks}, we see that the solution $\psi_{\tau}$ of equation \eqref {3.46} lies above solutions of the corresponding inequality.
\begin{lemma}[\protect\cite{laks2, laks}]
\label{Lemma3.43}For each fixed $\tau \in (\frac{1}{e},2]$, any nonnegative
solution $y$ of the differential inequality 
\begin{equation*}
y^{\prime }(t) \leq \left\{  \begin{array}{ll} \max \{x_{\tau }(\varrho (\tau )+t-\tau
),y(t)\}, & t\in [ 0,\tau ], \\ y(t),  & t\in [ \tau ,\infty ), \end{array} \right.   \quad y(0)=0,\notag \\
\end{equation*}
where $x_{\tau }$ is the solution of \eqref {03.55} and $\varrho $ is its
first root, described in Lemma {\ref{Lemma3.35}}, satisfies 
\begin{equation*}
y(t)\leq \psi _{\tau }(t),~~t\geq 0,
\end{equation*}%
where  $\psi _{\tau }$ is the solution of \eqref {3.46}. Moreover, the
function $t\mapsto \psi _{\tau }(t)-y(t)$, $t\geq 0$  is nondecreasing.
\end{lemma}

\begin{lemma}
$\Lambda :(\frac{1}{e},2]\rightarrow 
\mathbb{R}
$ is a continuous strictly decreasing function.
\end{lemma}

\begin{proof}
The result is known \cite[Lemma 3.4]{1} for the restriction to $[1,2]$. Let
us consider the restriction to $(\frac{1}{e},1].$ We first show
monotonicity. Assume $\frac{1}{e}<\tau _{1}<\tau _{2}\leq 1$. In virtue of
Lemma {\ref{Lemma3.35}}, Lemma {\ref{Lemma7.39}}, Lemma {\ref{Lemma3.43}}, $\psi _{\tau _{1}}(t)\leq
\psi _{\tau _{2}}(t),~t\geq 0$ and $\Lambda (\tau _{1})\geq \Lambda (\tau
_{2})$. We also have $\psi _{\tau _{1}}(t)<\psi _{\tau _{2}}(t),t\in (0,\xi
_{\tau _{1}}]\cap (0,\xi _{\tau _{2}}]$ by Corollary {\ref{Corollary3.19}}.
As $\psi _{\tau _{1}}^{\prime }(t)\leq \psi _{\tau _{2}}^{\prime }(t),t\in
[ 0,+\infty )$ we have the strict inequality $\psi _{\tau
_{1}}(t)<\psi _{\tau _{2}}(t),t>0$. This implies that $\Lambda (\tau
_{1})>\Lambda (\tau _{2})$.

Continuity follows from the continuity of $\varrho $ and the continuous
dependence on parameters for \eqref {03.55} and \eqref
{3.46}.
\end{proof}

\begin{lemma}
\label{Lemma10.46}Assume $|p(t)|\equiv 1$ and a fixed $\tau _{\max }\in (%
\frac{1}{e},2]$ and $x$ an $\alpha -$oscillating solution of \eqref {06.18}.
If 
\begin{equation*}
\alpha \leq \Lambda (\tau _{m }),
\end{equation*}%
where $\Lambda(\tau_{m})$ is defined in Definition {\ref{Definition4.40}},
then $x$ is bounded. If, further, $\alpha <\Lambda (\tau _{m})$, then $x$
tends to zero at infinity.
\end{lemma}

\begin{proof}
For $\tau _{m }\in [ 1,2]$ the result follows from \cite[Theorem 3.6%
]{1}. We shall proceed similarly in the case of $\tau _{m }\in (\frac{1}{e},1)$%
. We can assume that $\alpha >\max \{\Lambda (\tau _{m })-\varrho (\tau
_{m }),\varrho (\tau _{m })\}$, and that $\alpha $ is great enough so
that $\alpha -\varrho (\tau _{m})>\tau _{m }$, where $\varrho$ is the first root of $x_{\tau }$ the
solution of \eqref {03.55}, described in Lemma {\ref{Lemma3.35}}. Considering a point $t$ such that $x(t)=0$, we may without loss of generality assume 
$\underset{v\in [ t-\tau _{m}-\varrho (\tau _{m }),t]}{\sup }%
|x(v)|>0$. We define 
\begin{align}
z& =\inf\left\{ \zeta \in [ t,+\infty ):|x(\zeta )|>\psi _{\tau
}(\alpha -\varrho (\tau _{m }))\underset{v\in [ t-\tau _{m
}-\varrho (\tau _{m }),t]}{\sup }|x(v)|\right\} , \label{3.49}\\
\tilde{t}& =\sup \{\zeta \in [ t,z]:x(\zeta )=0\}, \notag \\
m& =\inf \{\zeta \in [ z,+\infty ):x(\zeta )=0\}, \notag
\end{align}%
where $\psi _{\tau }$ is defined in Lemma {\ref{Lemma4.31}}. Using Lemma {\ref{Lemma7.39}} and \eqref{06.18}, we have 
\begin{equation*}
|x^{\prime }|(t)\leq \max \left\{\underset{v\in [ t-\tau _{m}-\varrho
(\tau _{m}),t]}{\sup }|x(v)|x_{\tau _{m }}(\varrho (\tau _{m
})+t-\tau _{m}),\int_{\tilde{t}}^{t}|x^{\prime }|(w)dw \right\},~t\in [ 
\tilde{t},z].
\end{equation*}%
In virtue of Lemma {\ref{Lemma3.43}}, this inequality gives 
\begin{equation}
|x(q)|\leq \psi _{\tau _{m}}(q-\tilde{t})\underset{v\in [ t-\tau
_{m }-\varrho (\tau _{m }),t]}{\sup }|x(v)|,q\in [ \tilde{t},z]. \label{4.48}
\end{equation}%
Considering the definition of $z$ \eqref{3.49} and \eqref{4.48}, we have  \begin{equation}w-\tilde{t}>z-\tilde{t}\geq \alpha -\varrho (\tau _{m })>\tau
_{m }, \label{4.53}
\end{equation}%
 where \begin{equation*}|x(w)|=\max_{[\tilde{t},m]}|x|>\psi _{\tau _{m
}}(\alpha -\varrho (\tau _{m}))\underset{v\in [ t-\tau _{m
}-\varrho (\tau _{m}),t]}{\sup }|x(v)|,w\in [ \tilde{t},m].\end{equation*}
Applying Lemma {\ref{Lemma7.39}} and using \eqref{4.53}, $m-w\geq
\varrho (\tau _{m})$. But then $m-\tilde{t}>a,$ a contradiction.
\end{proof}

Noting that only consideration of the initial interval $[t-\tau _{m
}-\varrho (\tau _{m }),t]$ is necessary in the above proof as well as in
that of \cite[Theorem 3.6]{1} when $|p(t)|\equiv 1$ and $\tau _{m }\in
[ 1,2]$,  we have the following Corollary.

\begin{corollary}
\label{Corollary10.28}Assume that $x$ is an $\alpha$-oscillating solution
of \eqref {06.18} for $|p(t)|\equiv 1$ and a fixed $\tau _{m}\in (\frac{1}{e}%
,2]$. Consider a point $t_{0} $ such that $x(t_{0})=0$. Then $|x(t)|\leq 
\underset{v\in [ t_{0} -\tau _{m}-\varrho (\tau _{m}),t_{0} ]}{\max }%
|x(v)|,t\geq t_{0} $, where $\varrho $ is the first root of the
solution of $x_{\tau }$ to \eqref {03.55} described in Lemma {\ref{Lemma3.35}}.
\end{corollary}

\subsection{Sharpness of $\Lambda (\protect\tau )$}

The proof of the following Lemma is identical to that of {\cite[Theorem 2.1]{lillo}} and therefore is omitted.
\begin{lemma} [\protect {\cite[Theorem 2.1]{lillo}}]
\label{Lemma10.47}Let $x$ be an oscillatory solution of \eqref {06.18}, with $\limsup_{t\rightarrow \infty }|x(t)|>0,|p(t)|\equiv 1$.
Then for every $\varepsilon >0$
there exists an unbounded oscillatory function $y$, which possesses the same roots as $x$, solving \begin{equation*}
y'(t)=\overset{\sim }{p}(t)y(\tau(t)),
\end{equation*}
with 
\begin{eqnarray*}
1\leq|\overset{\sim }{p}(t)|\leq
1+\varepsilon, \quad
 {\rm sgn}[\overset{\sim }{p}(t)]\equiv {\rm sgn}[p(t)].
\end{eqnarray*}
\end{lemma}

The following Example, together with the method of Lemma {\ref{Lemma10.47}} and {\ref{appendixB}} Lemma {\ref{Lemma10.03}},
shows that the bounds in Lemma {\ref{Lemma10.46}} are sharp.

\begin{example}
\label{Example6.34} For a fixed $\tau \in (\frac{1}{e},1]$, define%
\begin{eqnarray*}
\varpi _{\tau }(t) &=&\psi _{\tau }(t),t\in [ 0,\Lambda (\tau
)-\varrho (\tau )] \\
\varpi _{\tau }(t) &=&x_{\tau }(t-\left( \Lambda (\tau )-\varrho (\tau
)\right) ),t\in [ \Lambda (\tau )-\varrho (\tau ),\Lambda (\tau )] \\
\varpi _{\tau }(t+\Lambda (\tau )) &=&\varpi _{\tau }(t), ~t\in 
\mathbb{R},
\end{eqnarray*}%
where $\varrho $ is the first root of $x_{\tau }$\ the solution of \eqref
{03.55}, described in Lemma {\ref{Lemma3.35}},  $\psi _{\tau }$ is the
solution of \eqref {3.46}, $\Lambda (\tau _{m})$ is defined in Definition~{%
\ref{Definition4.40}}. Then $\varpi _{\tau }$ is an $\Lambda (\tau )-$%
oscillating, $\Lambda (\tau )-$periodic solution of \eqref {06.18} with $%
|p(t)|\equiv 1,$ $\tau _{m}=\tau $.
\end{example}

\begin{corollary}
For fixed $\tau \in (\frac{1}{e},1]$ and arbitrary $\alpha >\Lambda (\tau )$%
, $r>\tau, $ there exist unbounded $\alpha -$oscillating solutions of %
\eqref{06.18},{\ w}ith $|p(t)|\equiv 1,$ $\tau _{m }\leq r$, where $%
\Lambda(\tau_{m})$ is defined in Definition~{\ref{Definition4.40}}.
\end{corollary}

\section{Main Results}
 From Lemma {\ref{Lemma7.39}} and the proof of Lemma {\ref{Lemma10.46}}, we obtain:

\begin{corollary}
\label{Corollary11.31}Assume that $x$ is an
oscillatory solution of \eqref {06.18} for $|p(t)|\equiv 1$ and a fixed $%
\tau _{m}\in (\frac{1}{e},2]$. Consider a point $t_{0} $ such that $x(t_{0} )=0$%
. Then 
\begin{equation*}
|x(t)|\leq \underset{v\in [ t_{0} -\tau _{m}-\varrho (\tau _{m}),t_{0} ]}{%
\max }|x(v)| \left\{  \begin{array}{ll} \psi _{\tau_{m} }(t-t_{0} ), & t\geq t_{0} , \\ x_{\tau_{m}
}(\varrho (\tau_{m} )-(t_{0} -t)), & t\in [t_{0}-\tau _{m},t_{0}] ,\end{array} \right.
\end{equation*}
where $\varrho $ is the first root of  the solution $x_{\tau_{m} }$ of \eqref {03.55}  described in Lemma~{\ref{Lemma3.35} }, $\psi _{\tau_{m} }$ is the
solution of \eqref {3.46}. Furthermore,%
\begin{align*}
|x^{\prime }(t)| & \leq   \underset{v\in [ t_{0}-\tau _{m}-\varrho (\tau
_{m}),t_{0}]} \max |x(v)|  \left\{  \begin{array}{ll} \psi _{\tau_{m} }(\tau (t)-t_{0}), & \tau
(t)\geq t_{0}, \\ x_{\tau_{m} }(\varrho (\tau )-(t_{0}-\tau (t))), & \tau (t)\leq
t_{0}, \end{array} \right.  \\ & \leq \underset{v\in [ t_{0}-\tau _{m}-\varrho (\tau _{m}),t_{0}]%
}{\max }|x(v)|\psi _{\tau_{m} }^{\prime }(t-t_{0}),~t\geq t_{0}.
\end{align*}
\end{corollary}

The following Lemma is the analog of Proposition {\ref{Proposition11.24}}, {%
\cite[Theorem 3.1, Theorem 3.3]{lillo}}, for the case of sign-changing
feedback. The broad strokes are inspired by the proof of {\cite[Theorem
3.1]{lillo}}.

\begin{lemma}
\label{Lemma10.50}
Let $x$ be nonnegative, $\Lambda (\tau )-$oscillating
solution of \eqref {06.18} with $|p(t)|\equiv 1,\tau _{m}=\tau \in (\frac{1}{%
e},2)$, where $\Lambda (\tau _{m})$ is defined in Definition {\ref%
{Definition4.40}}. Assume that 
\begin{eqnarray}
x(t_{0}) &=&0,  \notag\\
\mu  &=&\inf \{t>t_{0}:x(t)=x(\mu )\},  \notag \\
x(\mu ) &\geq &\left( \max_{[t_{0}-\tau -\varrho (\tau ),t_{0}]}x\right)
(1-\delta ),\mu >t_{0},  \label{6.32}
\end{eqnarray}%
where $\delta\in(0,1)$ is sufficiently small, $\varrho $ is the first root of the solution $x_{\tau }$ of \eqref {03.55}  described in Lemma~{\ref{Lemma3.35}}. 
Then,%
\begin{equation*}
\left\vert \frac{x(t)}{\max_{[t_{0}-\tau -\varrho (\tau
),t_{0}]}x }-\varpi _{\tau }(t-\tilde{t})\right\vert \leq O \left( \delta
^{2^{-1-\left\lfloor \frac{\varrho (\tau )}{\tau }\right\rfloor }} \right), ~t\in
[ \tilde{t}-\varrho (\tau ),\tilde{t}+\Lambda (\tau )],
\end{equation*}%
where $\tilde{t}:=\sup \{t<\mu :x(t)=0\}$ 
and $O(\delta ^{2^{-1-\left\lfloor \frac{%
\varrho (\tau )}{\tau }\right\rfloor }})$ depends on the fixed $\tau \in (%
\frac{1}{e},2).$
\end{lemma}

\begin{proof}
Define%
\begin{align*}
\tilde{t}& =\sup \{\zeta \in [ t_{0},\mu]:x(\zeta )=0\}, \\
m& =\inf\{\zeta \in [ \mu ,+\infty ):x(\zeta )=0\}, \\
w& =\inf \left\{ \zeta \in [ t_{0},+\infty ):x(\zeta )=\max_{[%
\tilde{t},m]}x\right\}.
\end{align*}%
By \eqref{6.32},
\begin{equation}
\max_{[ t_{0}-\tau -\varrho (\tau ),t_{0}]}x (1-\delta
)\leq x(w)\leq  \max_{[ t_{0}-\tau -\varrho (\tau
),t_{0}]}x . \label{6.33}
\end{equation}%
By Corollary {\ref{Corollary11.31}},
\begin{equation}
w-\tilde{t}\geq \psi _{\tau }^{-1}(1-\delta ).  \label{3.11}
\end{equation}%
where $\psi _{\tau }^{-1}:[0,\infty)\rightarrow[0,\infty)$ is the inverse function of $\psi _{\tau }$ defined in Lemma {\ref{Lemma4.31}}. Together with Lemma {\ref{Lemma7.39}}, this implies 
\begin{equation}
\varrho (\tau )-\delta \leq m-w\leq \Lambda (\tau )-\psi _{\tau
}^{-1}(1-\delta ).  \label{3.12}
\end{equation}%
By Definition {\ref%
{Definition4.40}}, $\psi _{\tau }^{\prime }(\Lambda (\tau )-\varrho (\tau ))=1,$ and considering the expansion of $\psi _{\tau }^{-1}$ near $1$, for small $\delta $ 
\begin{equation}
\Lambda (\tau )-\varrho (\tau )-\psi _{\tau }^{-1}(1-\delta )\leq \delta
+o(\delta ), \label{7.41}
\end{equation}%
where $o(\delta )$ depends on the fixed $\tau $. Together with \eqref{3.11}, %
\eqref{3.12}, this gives $|m-\varrho (\tau )-w|\leq \delta +o(\delta )$  and using \eqref{6.33},
\begin{equation}
x(m-\varrho (\tau ))\geq  (\max_{[ t_{0}-\tau -\varrho (\tau
),t_{0}]}x) (1-2\delta +o(\delta )).\label{7.44}
\end{equation}%
Applying Lemma {\ref{Lemma7.39}}, the last inequality gives
\begin{equation*}
0\leq x_{\tau }(t-m+\varrho (\tau ))-\frac{1}{\max_{[t_{0}-\tau
-\varrho (\tau ),t_{0}]}x }x(t)\leq   2\delta +o(\delta ), ~t\in
[ m-\varrho (\tau ),m].
\end{equation*}%
Furthermore, using \eqref{3.11},  
\eqref{3.12}, \eqref{7.41},  we get
\begin{equation}
\Lambda (\tau )-(m-\tilde{t})\leq\Lambda (\tau )-(m-w)-(w-\tilde{t}) \\
\leq \Lambda (\tau )-\varrho (\tau )+\delta +o(\delta )-\psi _{\tau
}^{-1}(1-\delta ) \\
\leq2\delta +o(\delta ). \label{7.42}
\end{equation}
Inequalities \eqref{7.44}, \eqref{7.42}, together with Lemma {\ref{Lemma3.43}, give for }$t\in [ \tilde{t},%
\tilde{t}{+}\Lambda (\tau )-\varrho (\tau )]$%
\begin{equation}
0 \leq \psi _{\tau }(t-\tilde{t})-\frac{x(t)}{
\max_{[t_{0}-\tau-\varrho (\tau),t_{0}]}x}
\leq 4\delta +o(\delta ).\label{12.41}
\end{equation}%
Now, by Lemma {\ref{Lemma7.39}}, for $t\in [ \tilde{t}-\varrho (\tau ),\tilde{t}]$%
\begin{equation*}
a(t):=\left( \max_{[t_{0}-\tau -\varrho (\tau ),t_{0}]}x\right)
x_{\tau }(\varrho (\tau )+t-\tilde{t})-x(t)
\end{equation*}%
is nonincreasing.

We distinguish between two cases. 
\vskip 1pt
I) $\tau \in [ 1,2)$. In this case, $\varrho(\tau)=1$. Assuming%
\begin{equation*}
a(\tilde{t}-1)\geq \gamma \left( \max_{[ t_{0}-\tau
-1,t_{0}]}x\right) >0,
\end{equation*}%
where $\gamma>0$, then using \eqref{06.18}, we obtain %
\begin{equation}
x(t)\leq \left( 1-\frac{\gamma }{2}\right) \left( \max_{[t_{0}-\tau
-1,t_{0}]}x\right) ,  ~t\in [ \tilde{t}-1,\tilde{t}]. 
\label{11.05}
\end{equation}%
In virtue of Remark {\ref{Remark10.36}}, for a fixed $\tau \in [ 1,2)$ we may consider $\gamma $ sufficiently
small so that $\gamma\in(0,\xi_{\tau}-(\tau -1))$ and hence
\begin{equation}
\psi' _{\tau }(t-\tilde{t})=1-(t-(\tilde{t}+\tau -1)),~t\in \left[  \tilde{t}+\tau -1,\tilde{t%
}+\tau -1+\frac{\gamma }{2} \right].\label{11.06}
\end{equation}%
We also assume $\gamma$ small enough so that
\begin{equation}
\psi _{\tau }(t-\tilde{t})\leq 1-\frac{\gamma }{2}, ~t\in \left[ \tilde{t}+\tau -1,\tilde{t%
}+\tau -1+\frac{\gamma }{2} \right].\label{11.07}
\end{equation}%
 Then \eqref{11.05}, \eqref{11.06}, \eqref{11.07}, in conjunction with Corollary {\ref{Corollary11.31}}, give
\begin{align*}
\frac{x^{\prime }(t)}{\max_{[t_{0}-\tau -1,t_{0}]}} & \leq \max \left\{ \frac{\max_{[ \tilde{t}-1,\tilde{t}]}x}{\max_{[t_{0}-\tau -1,t_{0}]}x},\psi _{\tau }(t-\tilde{t}) \right\}\leq \left( 1-\frac{\gamma }{2}\right)
\\ & \leq \psi' _{\tau }(t-\tilde{t})-(\tilde{t%
}+\tau -1+\frac{\gamma }{2}-t),~ t\in \left[ \tilde{t}+\tau -1,\tilde{t%
}+\tau -1+\frac{\gamma }{2} \right]
\end{align*}%
Integrating the last inequality and using Corollary {\ref{Corollary11.31}}, taking into account inequality \eqref{12.41}, we obtain 
\begin{equation*}
1-4\delta +o(\delta )\leq \frac{x(\Lambda (\tau )-\varrho (\tau )+\tilde{t})%
}{ \max_{[t_{0}-\tau -1,t_{0}]}x }\leq 1-\frac{1}{2}\left( \frac{\gamma }{2%
}\right) ^{2}.
\end{equation*}%
This immediately gives $\gamma \leq \sqrt{32\delta +o(\delta )}$.

II) $\tau \in (\frac{1}{e},1)$. Assuming%
\begin{equation*}
a(\tilde{t}-\tau )\geq \gamma \max_{[ t_{0}-\tau -\varrho
(\tau ),t_{0}]}x >0
\end{equation*}%
then by the Lipschiptian nature of $x,x_{\tau}$ ($x_{\tau}$ is of Lipschitz constant 1 and $x$ of Lipschitz constant  $\max_{[t_{0}-\tau -\varrho (\tau ),t_{0}]}x$), we have
\begin{equation*}
\frac{x(t)}{ \max_{[t_{0}-\tau -\varrho (\tau ),t_{0}]}x }\leq
x_{\tau }(\varrho (\tau )-(\tilde{t}-t))-\frac{\gamma }{2},  ~t\in  \left[ 
\tilde{t}-\tau ,\tilde{t}-\tau +\frac{\gamma }{4}  \right].
\end{equation*}%
Because $x_{\tau}$ is nonincreasing and of Lipschitz constant 1, this implies
\begin{eqnarray}
\frac{x(t)}{\left( \max_{[t_{0}-\tau -\varrho (\tau ),t_{0}]}x\right) } & \leq
\max \left\{ x_{\tau }(\varrho (\tau )-(\tilde{t}-t))-\frac{\gamma }{2},x_{\tau }(\varrho (\tau )-\tau+\frac{\gamma }{4}) \right\}  \nonumber \\ & \leq x_{\tau }(\varrho (\tau )-\tau+\frac{\gamma }{4}),
~~~t\in \left[ \tilde{t}-\tau,\tilde{t}  \right].
\label{10.52}
\end{eqnarray}%
For a fixed $\tau \in (\frac{1}{e},1)$ we may consider $\gamma $
sufficiently small so that also also $\gamma\in(0,\xi_{\tau})$ and using  $-1 \leq x'_{\tau}(t)\leq -x_{\tau}(\varrho(\tau )-\tau ),t\in[0,\varrho(\tau)],$
\begin{equation}
\psi' _{\tau }(t-\tilde{t})=x_{\tau } \left( \varrho (\tau  \right)-(\tilde{t}-t)-\tau )\geq x_{\tau } \left( \varrho (\tau )-\tau+\frac{\gamma }{4} \right)-\frac{\gamma }{8}x_{\tau}(\varrho(\tau )-\tau ),
~t\in \left[ \tilde{t},\tilde{t}+\frac{\gamma }{8} \right].
\label{10.53}
\end{equation}%
Lastly, we may (and do) assume $\gamma$ small enough so that
\begin{equation}
\psi _{\tau }(t-\tilde{t})\leq x_{\tau }\left( \varrho (\tau )-\tau+\frac{\gamma }{4} \right),~t\in \left[ \tilde{t},\tilde{t}+\frac{\gamma }{8} \right].
\label{10.54}
\end{equation}
By Corollary {\ref{Corollary11.31}}, and inequalities \eqref{10.52}, \eqref{10.53}, \eqref{10.54}, we have
\begin{equation*}
\frac{x^{\prime }(t)}{ \max_{[t_{0}-\tau -\varrho (\tau
),t_{0}]}x }\leq \psi' _{\tau }(t-\tilde{t})-\frac{\gamma }{8}x_{\tau}(\varrho(\tau )-\tau ),~t\in \left[ \tilde{t},\tilde{t}+\frac{\gamma }{8} \right].
\end{equation*}%
Integrating the last inequality and using Corollary {\ref{Corollary11.31}}, taking into account inequality \eqref{12.41}, we obtain 
\begin{equation*}
1-4\delta +o(\delta )\leq \frac{x(\Lambda (\tau )-\varrho (\tau )+\tilde{t})%
}{ \max_{[t_{0}-\tau -\varrho (\tau ),t_{0}]}x }\leq 1-\frac{%
\gamma ^{2}}{64}x_{\tau}(\varrho(\tau )-\tau ).
\end{equation*}%
This gives $\gamma \leq \sqrt{\frac{256}{x_{\tau}(\varrho(\tau )-\tau )}\delta +o(\delta )}$.

The result follows by considering Lemma {\ref{Lemma7.39}}.
\end{proof}
We notice that in the case $\tau_{m}=2$, the assumptions of Lemma {\ref{Lemma10.50}} need to be slightly modified. Essentially, the same bound holds, but only in a neighborhood of infinity, where the solution $x$ is uniformly bounded by $\limsup_{t\rightarrow \infty }|x(t)|+\varepsilon$. More precisely, we have the following statement:

\begin{lemma}
\label{Lemma10.49}Fix an $\varepsilon>0$. Let $x$ be a nonnegative, $2-$oscillating
solution of \eqref {06.18} with $|p(t)|\equiv 1,\tau _{m}=2$. Assume that 
\begin{eqnarray}
x(t_{0}) &=&0,  \notag\\
\mu  &=&\inf \{t>t_{0}:x(t)=x(\mu)\},  \notag \\
x(\mu ) &\geq &\left( \max_{[t_{0}-3-\varepsilon,t_{0}]}x\right)
(1-\delta ),\mu >t_{0},  \label{6.32}
\end{eqnarray}%
where $\delta\in(0,1)$ is sufficiently small. Then,%
\begin{equation*}
\left\vert \frac{x(t)}{\max_{[t_{0}-3-\varepsilon,t_{0}]}x }-\varpi _{\tau }(t-\tilde{t})\right\vert \leq O \left( \delta
^{2^{-1}} \right),~t\in
[ \tilde{t}-1,\tilde{t}+2],
\end{equation*}%
where $\tilde{t}:=\sup \{t<\mu :x(t)=0\}$.
\end{lemma}
\begin{proof}
The proof is identical to that of Lemma {\ref{Lemma10.50}} up until and including a slightly modified version of inequality \eqref{11.05}:
\begin{equation}
x(t)\leq \left( 1-\frac{\gamma }{2}\right) \left( \max_{[t_{0}-\tau
-1,t_{0}]}x\right) ,t\in \left[ \tilde{t}-1-\frac{\gamma }{2},\tilde{t} \right],
\label{12.38}
\end{equation}
where $\gamma$ is also assumed to be smaller than $\varepsilon$. Applying \eqref{12.38} and Corollary {\ref{Corollary11.31}}, we obtain
\begin{equation}
\frac{x^{\prime }(t)}{\max_{[t_{0}-3-\varepsilon,t_{0}]}}\leq \max \left\{ 1-\frac{\gamma }{2},\psi _{\tau }(t-\tilde{t}) \right\}
\leq  1-\frac{\gamma }{4}, ~~t\in \left[ \tilde{t}+1-\frac{\gamma }{2},\tilde{t%
}+1-\frac{\gamma }{4} \right].
\label{12.44}
\end{equation}
Integration of \eqref{12.44}, Corollary {\ref{Corollary11.31}}, and inequality \eqref{12.41} give
\begin{equation}
1-4\delta +o(\delta )\leq \frac{x(\Lambda (\tau )-\varrho (\tau )+\tilde{t})%
}{ \max_{[t_{0}-3-\varepsilon,t_{0}]}x }\leq 1-\frac{1}{16}\gamma^{2}.
\end{equation}
We conclude that $\gamma\leq\sqrt{64\delta+o(\delta)}$. The proof is complete.
\end{proof}
\begin{theorem}
\label{Theorem11.05} Let $x$ be a $\Lambda (\tau )-$oscillating solution of %
\eqref {06.18} with $|p(t)|\equiv 1,\tau _{m }\leq \tau \in (\frac{1}{e},2)$%
, where $\Lambda (\tau )$ is given by Definition {\ref{Definition4.40}}.
Then 
\begin{equation}
|x(t)|=M\varpi _{\tau }(t+\eta )+o(1),
\end{equation}%
where $\varpi _{\tau }$ is given in \eqref {14}, Example {\ref{Example6.34},}
$M=\limsup_{t\rightarrow \infty }|x(t)|$ and $\eta \in [ 0,\Lambda
(\tau ))$.
\end{theorem}

\begin{proof}
We can assume that $x$ is nonnegative (Lemma {\ref{Lemma11.30}}), $$M:=\limsup_{t\rightarrow \infty }|x(t)|>0,$$ and
select a $u_{0}$ such that $x(u_{0})=M$. Define for $n=0,1,...$,%
\begin{eqnarray*}
z_{n} &=&\inf \{t>u_{n}:x(t)=0\}, \\
u_{n+1} &=&\inf \{t>z_{n}:x(t)=M\},\\
m_{n} &=&\inf \{t\in [u_{n},z_{n}]:x(t)=\max_{[u_{n},z_{n}]}x\}, \\
w_{n} &=&\sup \{t<m_{n+1}:x(t)=0\}.
\end{eqnarray*}%
We first note that because $\tau<2$, we have $\Lambda(\tau)>\tau $. By considerations similar to Lemmata {\ref{Lemma10.46}},\text{ }{\ref{Lemma10.50}}, and by Corollary %
{\ref{Corollary10.28}}, we have that $x(m_{n})$ form a nonincreasing sequence, which tends to $M$. Hence,
the sequence $\left( x(m_{n+1})-x(m_{n})\right) $ is summable, and by Lemma {\ref{Lemma10.50}} (see inequality \eqref{7.42}), $\left( \Lambda (\tau
)-\left( z_{n+1}-w_{n}\right) \right) $ must also be summable. 
\vskip 1pt
We now need only show
that $(w_{n}-z_{n})$ is summable. Indeed, in that case, it suffices to take 
\begin{equation*}
\eta=\lim_{n\rightarrow\infty}\left( n\Lambda (\tau
)- z_{n} \right)-\Lambda (\tau)\left\lfloor \frac{\lim_{n\rightarrow\infty}\left( n\Lambda (\tau
)- z_{n} \right)}{\Lambda (\tau
)} \right\rfloor,
\end{equation*}
and use Lemma {\ref{Lemma10.50}}.
\vskip 1pt
By {Corollary \ref{Corollary10.28} we have that for any fixed }$\varepsilon>0${\ } 
\begin{equation}
w_{n}-z_{n}\leq \tau +\varepsilon  \label{8.39}
\end{equation}
holds eventually (for large $n$). By \eqref{8.39}, $\Lambda(\tau)>\tau $, and the proof of Lemma {\ref{Lemma10.46}}, we have that 
\begin{equation}
\max_{t\in
[ z_{n},w_{n}]}x(t)\leq M-k,
\label{11.22}
\end{equation} 
for a fixed $k>0$. 
Considering \eqref{8.39}, \eqref{11.22}, and Lemma {\ref{Lemma10.50}},  the sequence $(w_{n}-z_{n})$ must tend to zero. As noted in Remark {\ref{Remark10.36}}, there exists an interval where the delayed argument of the special solution is in the descent of the previous semicycle.
For a fixed $%
\varepsilon ^{\prime }>0$ sufficiently small ($\varepsilon ^{\prime }\in(0,\xi_{\tau}-(\tau -1))$ for $\tau\in[1,2)$, or $\varepsilon ^{\prime }\in(0,\xi_{\tau})$ for $\tau\in(1/e,1)$), consider the derivative on this interval, $%
\left[ w_{n}+\tau -1,w_{n}+\tau -1+\varepsilon^{\prime } \right] $ for $\tau \in
[ 1,2)$ or $\left[ w_{n},w_{n}+\varepsilon^{\prime } \right] $ for $\tau \in (%
\frac{1}{e},1)$. Using Corollary {\ref{Corollary11.31}} and the fact that $(w_{n}-z_{n})$ tends to zero, we get
\begin{align*}
\frac{x^{\prime }(t)}{x(m_{n})} & \leq \max\{\frac{\max_{
[ z_{n},w_{n}]}x}{x(m_{n})},\psi _{\tau }(t-w_{n}),x_{\tau }(\varrho (\tau )-(z_{n}-t)-\tau )\} \\ & =x_{\tau }(\varrho (\tau )-(z_{n}-t)-\tau )\leq \psi _{\tau }^{\prime
}(t-w_{n})-C\left( w_{n}-z_{n}\right),
\end{align*}%
where $C=x_{\tau }(\varrho (\tau )-\tau
)=\min_{t\in[0,\varrho(\tau)]}|x'_{\tau}(t)|$. Integrating the last inequality and using Corollary~{\ref{Corollary11.31}}, we obtain 
$\left( 1-\frac{x(m_{n+1})}{x(m_{n})}\right) \geq\varepsilon ^{\prime }C\left( w_{n}-z_{n}\right)$. We conclude that, in order for 
 $\left( x(m_{n+1})-x(m_{n})\right) $ to be summable, $(w_{n}-z_{n})$ must also be summable.
\end{proof}

\section{Discussion of the Buchanan-Lillo Conjecture}

\bigskip In parallel to the research of Myshkis, Buchanan-Lillo, and
starting with the Morse decompositions \cite{mallet} for the autonomous
case, oscillation frequency was also measured in terms of the discrete
Lyapunov functions introduced by Mallet-Paret and Sell \cite{malletsell}.
Two major conclusions may be drawn from this vein of research. Firstly, that
the rate of growth/decay of oscillatory solutions was again linked to the
oscillation frequency, a bound on the oscillation speed implying a bound on
the rate of decay \cite{cao1, cao2, garab, krisztin, kriwalther}. Secondly,
that oscillatory solutions are asymptotic to an attractor with separated
zeros \cite{garab, krisztin, mallet, malletsell}. However, these asymptotic
results, while applicable to a seemingly more general class of equations,
require a strong regularity of the parameters, such as monotonicity or
continuity. This is not admissible in the present investigations, as the
equations which describe the threshold periodic solutions have, as an
essential property, discontinuities in the delay \cite[Corollary 3.4]{lillo}%
. Hence, it is not obvious how to extend these results (that oscillatory
solutions are asymptotic to a solution with separated zeros) to equations
with measurable parameters. Finally, while these methods are intimately
linked to periodic solutions in the autonomous case, in virtue of the Poincar%
\'{e}-Bendixson Theorem \cite{kennedy, malletsellb}, such a connection in
the nonautonomous case is not a direct consequence of known results. This
contrasts with the methods of Myshkis, Buchanan-Lillo, Stavroulakis and
Braverman, which, at their core, reason on basis of periodic solutions
situated on the threshold between boundedness and unboundedness. While one
may hope that the two approaches can be united in a broader Theory, the
source and reason behind the periodic solutions crucial to the results of
Myshkis, Buchanan-Lillo, remain an enigma.

{Furthermore, all the known results on the positive feedback case }\cite%
{buch, buch1971, buch1974, lillo}, following Buchanan's Thesis \cite{buch},
consider sequences of extrema of sufficiently regular solutions. This
renders the comparison of arbitrary oscillatory solutions to $\varpi ^{+}$
and the proof of the Buchanan-Lillo Conjecture quite difficult. It should be
noted that the techniques of Lillo \cite{lillo} are in the same vein, and
moreover explicitly require the unboundedness of solutions, hence are not
directly applicable.

One hopes that the proof of the conjecture may be undertaken in two steps:
firstly reducing the problem to oscillatory solutions with separated zeros
(cf. \cite{garab, krisztin, mallet, malletsell}), and secondly generalizing
Buchanan's method \cite[p. 48]{buch}, \cite[Theorem 3]{buch1971}, \cite%
{buch1974}. However both steps are elusive given current knowledge. The
accidental character of the specific rate of convergence and its dependence
on apparently haphazard properties of the periodic solution is also
noteworthy. If a more general framework can account for both positive and
negative feedback (Buchanan-Lillo Conjecture and Proposition {\ref%
{Proposition11.24}), then either a weaker rate of convergence would be
obtained, or the theory would have to be significantly more intricate than
the methods employed until now.}

The Buchanan-Lillo Conjecture represents the last remaining open question
about the functions $\varpi ^{+}$ and $\varpi ^{-}$ of Myshkis and Soboleva \cite[%
p. 171]{myshkis72}, and their relation to the critical state of bounded
oscillatory solutions. A resolution of this conjecture would provide a more
complete overview and understanding of the research framework that began
with the $3/2$-criterion \cite{myshkispaper51} in 1951. The authors trust
that its eventual (dis)proof will give deep insights into these equations.

\appendix
\appendixpage
\addappheadtotoc

\section{Asymptotics of $\Lambda (\protect\tau )$ near $\frac{1}{e}^{+}$%
\label{appendixA}}

We calculate the asymptotics of $\Lambda (\tau )$ (given in Definition {\ref{Definition4.40}}) near $\frac{1}{e}^{+}$,
utilizing the following Lemmata, the proof of which can be found in \cite{2}.

\begin{definition}
\label{Definition6.01}
For a fixed $\tau \in (\frac{1}{e},+\infty )$ we define $\mu (\tau ),\nu
(\tau ),\gamma (\tau )$ by%
\begin{eqnarray*}
\mu (\tau )+\tau \exp (-\mu (\tau ))\cos \nu (\tau ) &=&0 \\
\nu (\tau )-\tau \exp (-\mu (\tau ))\sin \nu (\tau ) &=&0 \\
\gamma (\tau )=\arctan \frac{1+\mu (\tau )}{\nu (\tau )} &\in & \left( 0,\frac{\pi 
}{2} \right).
\end{eqnarray*}
\end{definition}

\begin{lemma}[{\cite[Proposition 2, Theorems 5,6, proof of Theorem 6]{2}}]
\label{Lemma4.23} The functions $\mu :(\frac{1}{e},+\infty )\rightarrow
(-1,+\infty )$, $\nu : (\frac{1}{e},+\infty )\rightarrow (0,\pi )$ described above in Definition~{\ref{Definition6.01}}, are
continuous strictly decreasing and satisfy the following asymptotic
relations 
\begin{eqnarray*}
\lim_{\tau \rightarrow \frac{1}{e}^{+}}\frac{1+\mu (\tau )}{\nu (\tau )} &=&0
\\
\lim_{\tau \rightarrow \frac{1}{e}^{+}}\nu (\tau ) &=&0 \\
\lim_{s\rightarrow \frac{1}{e}^{+}}\mu (\tau ) &=&-1 \\
\lim_{\tau \rightarrow \frac{1}{e}^{+}}\frac{1+\mu (\tau )}{\nu ^{2}(\tau )}
&=&\frac{1}{3}.
\end{eqnarray*}%
Furthermore, 
\begin{eqnarray*}
\lim_{\tau \rightarrow \frac{1}{e}^{+}}\gamma (\tau ) &=&0, \\
\lim_{\tau \rightarrow \frac{1}{e}^{+}}\frac{\gamma (\tau )}{\nu (\tau )} &=&%
\frac{1}{3}
\end{eqnarray*}%
and 
\begin{equation*}
\lim_{\tau \rightarrow \frac{1}{e}^{+}}\left[ \frac{\varrho (\tau )+\tau }{%
\tau }-\left( \frac{\pi }{\nu (\tau )}-\frac{1}{3}\right) \right] =0,
\end{equation*}
where  $\varrho $
is the first root of the solution  $x_{\tau }$ of \eqref {03.55} described in Lemma~{\ref{Lemma3.35}}.
\end{lemma}

\begin{lemma}[{\cite[proof of Theorem 6]{2}}]
\label{Lemma4.26}For $\tau $ near $\frac{1}{e}^{+},t\in [ 0,\tau ]$%
\begin{eqnarray*}
& \left| x_{\tau }(\varrho (\tau )+t-\tau )-\exp \left\{ \mu (\tau ) \left( \frac{\varrho
(\tau )+\tau }{\tau }-\left( \frac{\tau -t}{\tau }\right)  \right) \right\} \times  \right.
\\
& \times  \left. \frac{2}{%
\sqrt{\left( \frac{1+\mu (\tau )}{\nu (\tau )}\right) ^{2}+1}}\frac{\sin
(\nu (\tau )(\frac{\varrho (\tau )+\tau }{\tau }-\left( \frac{\tau -t}{\tau }%
\right) )+\gamma (\tau ))}{\nu(\tau )} \right| \\
\leq &\frac{\pi ^{2}}{3(\ln 4-1)}\exp \left\{ -(\frac{\varrho (\tau )+\tau }{\tau }%
- \frac{\tau -t}{\tau } )(\ln 4-1) \right\} \exp  \left\{ \mu (\tau )(\frac{%
\varrho (\tau )+\tau }{\tau }-\left( \frac{\tau -t}{\tau }\right) )\right\},
\end{eqnarray*}
where $x_{\tau }$ is the solution of \eqref {03.55}, $\varrho $
is its first root, described in Lemma {\ref{Lemma3.35}}, and the functions $\mu,\nu,\gamma$ are given in Definition {\ref{Definition6.01}} and Lemma {\ref{Lemma4.23}}.
\end{lemma}

\begin{theorem}
For $\tau $ near $\frac{1}{e}^{+},$ 
\begin{equation*}
\Lambda (\tau )=\frac{\pi }{\sqrt{2e^{3}}} \left( 1+\frac{1}{e} \right)\left( \tau -\frac{1%
}{e}\right) ^{-1/2}+o \left(\left( \tau -\frac{1}{e}\right) ^{-1/2} \right).
\end{equation*}
\end{theorem}

\begin{proof}
\bigskip From Lemmata {\ref{Lemma4.23}},\text{ }{\ref{Lemma4.26}}, we obtain using  $\lim\limits_{x\rightarrow0}\frac{\sin x}{x}=1$ successively  
\begin{eqnarray*}
&&\left| x_{\tau }(\varrho (\tau )+t-\tau )-\exp \left\{ \left[ -1+ \left(\frac{1}{3}%
+o(1) \right)\nu ^{2}(\tau )\right] \left( \frac{\pi }{\nu (\tau )}-\frac{1}{3}%
+o(1)-\left( \frac{\tau -t}{\tau }\right) \right) \right\} \times  \right. \\
&& \left. \times 2(1+o(1))\frac{\sin (v(\tau )(\frac{\pi }{\nu (\tau )}-\frac{1}{3}%
+o(1)-\left( \frac{\tau -t}{\tau }\right) )+\gamma (\tau ))}{v(\tau )} \right| \\
&\leq &\frac{\pi ^{2}}{3(\ln 4-1)}\exp \left\{ - \left( \frac{\pi }{\nu (\tau )}-%
\frac{1}{3}+o(1)-\left( \frac{\tau -t}{\tau }\right) )(\ln 4-1)\right) \right\} 
\times  \\
&&\exp \left\{ \left[ -1+(\frac{1}{3}+o(1))\nu ^{2}(\tau )\right] \left( 
\frac{\pi }{\nu (\tau )}-\frac{1}{3}+o(1)-\left( \frac{\tau -t}{\tau }%
\right) \right) \right\}, 
\end{eqnarray*}%
\begin{eqnarray*}
&& \left|x_{\tau }(\varrho (\tau )+t-\tau )-\exp \left\{ -\frac{\pi }{\nu (\tau )}+\frac{1%
}{3}+\left( \frac{\tau -t}{\tau }\right)  \right\} 2(1+o(1))\frac{\sin (\pi -\left( 
\frac{\tau -t}{\tau }+o(1)\right) \nu (\tau ))}{\nu (\tau )}\right| \\
&\leq &O \left(\exp \left\{ \frac{-\pi }{\nu (\tau )}\ln 4 \right\} \right),
\end{eqnarray*}%
\begin{eqnarray}
x_{\tau }(\varrho (\tau )+t-\tau )=& \exp  \left\{ \frac{-\pi }{\nu (\tau )} \right\} \exp \left\{
\frac{1}{3}+ \frac{\tau -t}{\tau } \right\} (2+o(1))\left( \frac{\tau
-t}{\tau }+o(1)\right) \nonumber \\ & + O\left( \exp \left\{ \frac{-\pi }{\nu (\tau )}\ln 4 \right\} \right)  \label{41.31}
\end{eqnarray}%
Now we calculate $\xi _{\tau }$, the solution of \eqref {4.32}, described in Lemma {\ref{Lemma4.31}}. We
first estimate the left-hand side of \eqref {4.32}. The calculation is straightforward integration of \eqref {41.31}, and we also use the fact that $\nu$ tends to zero, hence $\exp\{ -(\ln4-1)\frac{\pi}{\nu(\tau)}\}=o(1)$. 
Thus,
\begin{eqnarray*}
&&\int_{0}^{\xi _{\tau }}x_{\tau }(\varrho (\tau )+t-\tau )dt \\
&=&\int_{0}^{\xi _{\tau }}\exp  \left\{ -\frac{\pi }{\nu (\tau )} \right\} \exp \left\{ \frac{1}{3%
}+\left( \frac{\tau -t}{\tau }\right)  \right\} 2(1+o(1))\left( \frac{\tau -t}{\tau }%
+o(1)\right)dt \\ & &  +\int_{0}^{\xi _{\tau }}O\left( \exp  \left\{ -\frac{\pi }{\nu (\tau )}\ln
4 \right\} \right) \, dt \\
&=&2\exp \left\{ -\frac{\pi }{\nu (\tau )} \right\} \exp \left\{ \frac{1}{3}\right\} \int_{0}^{\xi _{\tau
}}\exp \left\{ \frac{\tau -t}{\tau }\right\} \left( \frac{\tau -t}{\tau }+o(1)\right)dt
+O\left( \exp\left\{-\frac{\pi }{\nu (\tau )}\ln 4 \right\} \right)  \\
&=&2e\exp \left\{ -\frac{\pi }{\nu (\tau )}\right\} \exp \left\{ \frac{1}{3} \right\} \int_{0}^{\xi _{\tau
}}\exp \left\{  \frac{\tau -t}{\tau }\right\}  \left( \tau -t\right)dt +o \left(\exp \left\{ -%
\frac{\pi }{\nu (\tau )}\right\} \right) \\
&=&2e\exp\left\{ -\frac{\pi }{\nu (\tau )}\right\} \exp \left\{ \frac{1}{3} \right\} \exp \left\{ 1-e\xi _{\tau
} \right\} \xi _{\tau }+o(1)\exp  \left\{ -\frac{\pi }{\nu (\tau )}) \right\}.
\end{eqnarray*}%
By expression \eqref {41.31}, we immediately see that the right-hand side of \eqref {4.32} has the
asymptotic expansion 
\begin{eqnarray*}
&&\exp \left\{ -\frac{\pi }{\nu (\tau )} \right\} \exp \left\{ \frac{1}{3}+\left( \frac{\tau -\xi
_{\tau }}{\tau }\right)  \right\} 2(1+o(1))\left( \frac{\tau -\xi _{\tau }}{\tau }%
+o(1)\right).
\end{eqnarray*}%
Therefore $\xi _{\tau }$ =$\frac{1+o(1)}{2e}.$ Now, using Lemma {\ref{Lemma3.35}},
Lemma {\ref{Lemma4.31}, Definition \ref{Definition4.40}}, Lemma {\ref{Lemma4.23}}, we can
calculate the asymptotics of $\Lambda $:
\begin{eqnarray*}
\Lambda (\tau ) &=&\varrho (\tau )+1/(2e)+o(1)-\ln \left( 2e\exp \left\{ -\frac{\pi }{\nu (\tau )%
} \right\} \exp\left\{ \frac{1}{3} \right\} \exp \left\{ 1-e\xi _{\tau } \right\} \xi _{\tau }(1+o(1)) \right) \\
&=&\varrho (\tau )+\frac{\pi }{\nu (\tau )}+O(1) \\
&=&\frac{\varrho (\tau )+\tau }{\tau }+\varrho (\tau )+O(1) \\
&=&\frac{\pi }{\sqrt{2e^{3}}}(1+\frac{1}{e})\left( \tau -\frac{1}{e}\right)
^{-1/2}+o \left( \left( \tau -\frac{1}{e}\right) ^{-1/2} \right).
\end{eqnarray*}
\end{proof}

\section{Change of variables\label{appendixB}}

We prove a generalization of a well-known change of variables \cite{ladas} to the case where the coefficient (prior to the change of variables) may be zero on sets of positive measure. This allows us to assume $|p(t)|\equiv1$. 

\begin{lemma}
\label{Lemma10.03}Define%
\begin{equation}
f(t):=\left\{  \begin{array}{ll} \displaystyle \int_{0}^{t}|p(s)|ds, & t\geq 0\\ \\ t, & t\leq 0 \end{array} \right.
\label{6.49}
\end{equation}%
Assume that $f(+\infty )=+\infty $ and that $x$ is a solution of %
\eqref{06.18}, with initial point $t=0$. Define the strictly increasing
function 
\begin{equation*}
g(t):=\inf \{s\geq 0\colon f(s)=t\},\,\,\,\,t\in [ 0,\infty ),
\end{equation*}%
where $f$ is given in \eqref{6.49}, and a function $\tilde{x}$ by 
\begin{equation*}
\tilde{x}(s)=x(t)\text{ if and only if }s=f(t).
\end{equation*}%
Then, $|p\left( g(s)\right) |>0,$ a.e., and $\tilde{x}$ is a solution of 
\begin{equation}
\tilde{x}^{\prime }(s)= {\rm sgn}[p\left( g(s)\right) ]\tilde{x}(f(\tau
(g(s)))),s\geq 0  \label{2.05}
\end{equation}%
where $sgn(\cdot)$ is the sign function. Equivalently, 
\begin{equation}
\tilde{x}^{\prime }(s)= {\rm sgn}[p\left( g(s)\right) ]\,\,\left\{ 
 \begin{array}{ll} \displaystyle %
\tilde{x}  \left(  s-\int_{\tau (g(s))}^{g(s)}|p(z)|dz\right) ,  & \tau (g(s))\geq
0 \\ \tilde{x}  [ \tau (g(s))  ] , & \tau (g(s))\leq 0, \end{array} \right. \quad s\geq 0.
\label{2.05a}
\end{equation}
\end{lemma}

\begin{proof}
We first prove that $\tilde{x}$ is well-defined and  locally Lipschitz
on $[0,+\infty )$. In the following, $0\leq u<v$. Using the fact
that $x$ is a solution on $[0,+\infty )$, we derive 
\begin{align*}
|x(v)-x(u)|& =\left\vert \int_{u}^{v}p(s)x(\tau (s))ds\right\vert \\
& \leq \left\vert \int_{u}^{v}|p(s)|ds\right\vert {\rm esssup}_{s\in [ u,v]}|x(\tau (s))| \\
& \leq |f(u)-f(v)|{\rm esssup}_{s\in [ u,v]}|x(\tau (s))|.
\end{align*}%
For any points $t,s\in {\mathbb{R}}_{+}$, $g(t)=s$ holds if and only if $%
t=f(s)$, $g(f(s))=s$. Therefore, for any measurable set $A\subset {\mathbb{R}%
}_{+}$, we have $g^{-1}(A):=\{x\in {\mathbb{R}}_{+}:g(x)\in A\}=f(A\cap g({%
\mathbb{R}}_{+}))$ and $g({\mathbb{R}}_{+})=\{s\in {\mathbb{R}}_{+}\colon
g(f(s))=s\}.$ As $g({\mathbb{R}}_{+})$ is measurable, by absolute continuity
of $f$, we have that $g^{-1}(A)$ is measurable. Hence, when $r$ is
measurable, and $r\circ g$ is well-defined, $r\circ g$ is measurable (so for
example, the function $|p\left( g(s)\right) |$ is measurable).

In view of the properties of $g$ and $f$, the composite function $g\circ f$
is defined and monotone on ${{\mathbb{R}}}_{+}$ and can have countably many
discontinuities. Also, considering the definition of $g$, $g(f(s))\leq
s,\forall s\geq 0.$ Consider the sets 
\begin{align*}
S_{uv}& :=(u,v)\setminus g({\mathbb{R}}_{+})=\{s\in (u,v)\colon g(f(s))<s\},
\\
D_{uv}& :=\{s\in S_{uv}\colon g\circ f\ \ {\text{is discontinuous at }}s\},
\\
C_{uv}& :=\{s\in S_{uv}\colon g\circ f\ \ {\text{is continuous at }}s\}.
\end{align*}%
Then $D_{uv}$ is at most countable. Also, 
\begin{equation*}
\forall z\in C_{uv}\,,\frac{1}{2}\left( z-g(f(z))\right) >0.
\end{equation*}%
By continuity, $\exists \,\varepsilon (z)\in (0,\min [z-u,v-z])\colon $ 
\begin{equation*}
s\in (z-\varepsilon (z),z+\varepsilon (z))\Longrightarrow g(f(s))-s<\frac{1}{%
2}\left( g(f(z))-z\right) <0.
\end{equation*}%
Hence, 
\begin{equation*}
(z-\varepsilon (z),z+\varepsilon (z))\subset S_{uv},\,\,\forall z\in C_{uv}
\end{equation*}%
Now, 
\begin{equation*}
\left( \bigcup_{z\in C_{uv}}(z-\varepsilon (z),z+\varepsilon (z))\right)
\cup D_{uv}=S_{uv},
\end{equation*}%
where $\bigcup_{z\in C_{uv}}(z-\varepsilon (z),z+\varepsilon (z))$ is an
open set in ${\mathbb{R}}$ (in the standard topology) and it can be written
as a collection of mutually disjoint open intervals $\{I_{\beta }\}_{\beta
\in {\mathbb{B}}}$, where ${\mathbb{B}}\subset {\mathbb{N}}$. We can write 
\begin{equation*}
S_{uv}=\left( \bigcup_{\beta \in {\mathbb{B}}}I_{\beta }\right) \cup D_{uv}.
\end{equation*}%
Furthermore, we have 
\begin{equation*}
\int_{I_{\beta }}\mu (s)|p(s)|ds=0,\,\,\,\,\forall \beta \in {\mathbb{B}}
\end{equation*}%
because we will show that $I_{\beta }\subset S_{uv}$ implies that $f$ is
constant on $I_{\beta },$ i.e. $f(\inf  I_{\beta })=f(\sup I_{\beta })$
and 
\begin{equation*}
f(\sup I_{\beta })-f(\inf   I_{\beta })=\int_{I_{\beta }}\mu
(s)|p(s)|ds=0.
\end{equation*}%
This is easy to see. In fact, $\forall t\in S_{uv}$, $t>g(f(t))$ and $%
g(f(t))\in g({\mathbb{R}}_{+})$. Considering the definition of $S_{uv}$ we
have 
\begin{equation*}
g(f(t))\notin S_{uv}\supset I_{\beta },\forall t\in S_{uv}.
\end{equation*}%
For any fixed $\beta \in {\mathbb{B}}$, we can therefore assume that $%
g(f(z))\leq \inf  I_{\beta }$, $\forall z\in I_{\beta }$. Using the fact
that $f(g(f(z)))=f(z)$, i.e., $f$ is constant on $[g(f(z)),z]\supset [ 
\inf  I_{\beta },z]$, it is easy to see that $f$ is constant on $%
I_{\beta }=\cup _{z\in I_{\beta }}(\inf  I_{\beta },z]$. Hence, setting
for any set $A\subset {\mathbb{R}}$, 
\begin{equation*}
\chi (A,s)=%
\begin{cases}
1{\text{ if }}s\in A, \\[2mm] 
0{\text{ if }}s\in {\mathbb{R}}\setminus A, %
\end{cases}%
\end{equation*}%
we have 
\begin{equation}
\int_{u}^{v}\chi \left( S_{uv},w\right) |p(w)|dw=0  \label{1.5}
\end{equation}%
and 
\begin{equation}
\chi \left( S_{uv},w\right) |p(w)|=0\,\,\,{\text{ a.e. for }}\,\,w\in
[ u,v].  \label{1.6}
\end{equation}%
Now, we show that $|p\left( g(s)\right) |>0$ a.e.

We use the change of variables theorem for the Lebesgue integral (\cite%
{serrin}, \cite{tandra}) 
\begin{align*}
I & =\int_{f(u)}^{f(v)}\chi \left( \left\{ s\in [ f(u),f(v)]\colon
|p\left( g(s)\right) |>0\right\} ,w\right) dw \\ & =\int_{u}^{v}\chi \left(
\left\{ s\in [ u,v]\colon |p\left( g(f(s))\right) \}|>0\right\}
,w\right) p(w)dw.
\end{align*}%
Obviously, 
\begin{eqnarray*}
&&  \chi \left( \left\{ s\in [ u,v]\colon |p\left( g(f(s))\right)
|>0\right\} ,w\right) \\ &= & \chi \left( \left\{ s\in [ u,v]\colon |p\left(
s\right) |>0\right\} ,w\right) \chi \Big (\Big \{s\in [ u,v]\colon
g(f(s))=s\Big \},w\Big ) \\
& & +\chi \left( \left\{ s\in [ u,v]\colon |p\left( g(f(s))\right)
|>0\right\} ,w\right) \chi \Big (\Big \{s\in [ u,v]\colon g(f(s))\neq s%
\Big \},w\Big ). 
\end{eqnarray*}%
By formula~\eqref {1.6}, 
\begin{equation*}
\int\limits_{u}^{v}\chi \left( \left\{ s\in [ u,v]\colon |p\left(
g(f(s))\right) |>0\right\} ,w\right) \chi \Big (\Big \{s\in [
u,v]\colon g(f(s))\neq s\Big \},w\Big )|p(w)|dw=0.
\end{equation*}%
Then 
\begin{equation*}
I=\int_{u}^{v}\chi \left( \left\{ s\in [ u,v]\colon |p\left( s\right)
|>0\right\} ,w\right) \cdot \chi \Big (\Big \{s\in [ u,v]\colon
g(f(s))=s\Big \},w\Big )|p(w)|dw.
\end{equation*}%
Since 
\begin{equation*}
\int_{u}^{v}\chi \left( \left\{ s\in [ u,v]\colon |p\left( s\right)
|\leq 0\right\} ,w\right) \cdot \chi \Big (\Big \{s\in [ u,v]\colon
g(f(s))=s\Big \},w\Big )|p(w)|dw=0,
\end{equation*}%
we have 
\begin{equation*}
I=\int_{u}^{v}\chi \Big (\Big \{s\in [ u,v]\colon g(f(s))=s\Big \},w%
\Big )|p(w)|dw.
\end{equation*}%
Further simplification gives (we use formula~\eqref {1.6} again) 
\begin{eqnarray*}
I & = & \int_{u}^{v}\chi \Big (\Big \{s\in [ u,v]\colon g(f(s))=s\Big \},w%
\Big )|p(w)|dw  \\
& & +\int_{u}^{v}\chi \Big (\Big \{s\in [ u,v]\colon g(f(s))\neq s\Big \},w%
\Big )|p(w)|dw  \\
& = & \int_{u}^{v}|p(w)|dw=f(v)-f(u)=\int_{f(u)}^{f(v)}1\,dw. 
\end{eqnarray*}%
This implies that almost everywhere 
\begin{equation}
\chi \left( \{s\in [ f(u),f(v)]\colon |p\left( g(s)\right)
|>0\},w\right) =1,\,\,\,w\in [ f(u),f(v)]  \label{7.28}
\end{equation}%
and almost everywhere 
\begin{equation*}
|p\left( g(s)\right) |>0,\,\,\,\,s\in [ 0,+\infty ).
\end{equation*}%
Also, the function 
\begin{equation*}
\frac{p\left( g(s)\right) }{\displaystyle|p\left( g(s)\right) |}%
\,,\,\,\,\,s\in [ 0,+\infty )
\end{equation*}%
is measurable and bounded. Consider the integral 
\begin{align*}
J:=& \int_{f(u)}^{f(v)}\frac{p\left( g(s)\right) }{\displaystyle|p\left(
g(s)\right) |}{\tilde{x}}(f(\tau (g(s))))ds \\
=& \int_{f(u)}^{f(v)}\frac{p\left( g(s)\right) }{\displaystyle|p\left(
g(s)\right) |}x(\tau (g(s)))ds.
\end{align*}%
\vspace*{10mm} Then also 
\begin{equation*}
J=\int_{u}^{v}\frac{p\left( g(f(s))\right) }{|p\left( g(f(s))\right) |}%
x(\tau (g(f(s))))|p(s)|ds
\end{equation*}%
The following obviously holds:
\begin{align*}
J=& \int_{u}^{v}\Bigg (\frac{p(g(f(s)))x(\tau (g(f(s))))}{|p\left(
g(f(s))\right) |}|p(s)|\chi ^{\ast }(s) \\
& +\frac{p\left( g(f(s))\right) x(\tau (g(f(s))))}{|p\left( g(f(s))\right) |}%
|p(s)|\chi ^{\ast \ast }(s)\Bigg )ds,
\end{align*}%
where 
\begin{align*}
\chi ^{\ast }(s)& :=\chi \Big (\Big \{w\in [ u,v]\colon g(f(w))=w\Big
\},s\Big ), \\
\chi ^{\ast \ast }(s)& :=\chi \Big (\Big \{w\in [ u,v]\colon
g(f(w))\neq w\Big \},s\Big ).
\end{align*}%
Moreover, applying~\eqref {1.6}, we continue to compute $J$, 
\begin{eqnarray*}
J &=&\int_{u}^{v}p(s)x(\tau (s))\chi ^{\ast }(s)ds \\
&=&\int_{u}^{v}\Bigg (p(s)x(\tau (s))\chi ^{\ast }(s)+p(s)x(\tau (s))\chi
^{\ast \ast }(s)\Bigg )ds \\
&=&\int_{u}^{v}p(s)x(\tau (s))ds=x(v)-x(u)={\tilde{x}}(f(v))-{\tilde{x}}%
(f(u)).
\end{eqnarray*}%
It is easy to see that the last line is equivalent to~\eqref{2.05}, %
\eqref{2.05a}, the two expressions being equal by the relation $%
f(g(s))=s,\forall s\geq 0$.
\end{proof}

\section*{Acknowledgments}

We thank Anatoli Ivanov for valuable consultations on the properties of the
periodic solutions and the overall presentation of the results. We thank 
Ábel Garab and Tibor Krisztin for insightful discussions regarding the
assumptions under which oscillatory solutions are asymptotic to attractors
with separated zeros, and possible extensions of such results.
The first author was supported by the NSERC Grant RGPIN-2020-03934. 
The second author acknowledges the support of Ariel University.

\end{document}